\documentclass[10pt,a4paper]{amsart}

\usepackage[utf8]{inputenc}
\usepackage{comment}
\usepackage[T1]{fontenc}
\usepackage{mathtools}
\usepackage[a4paper, margin=2.7cm]{geometry}

\usepackage{amsmath, amsthm, amsfonts, amssymb,bm}
\usepackage{bbm}
\usepackage{mathrsfs}  
\usepackage{enumitem}
\usepackage[colorlinks=true,linkcolor=blue,citecolor=blue,urlcolor=blue,breaklinks]{hyperref}

\usepackage{graphicx}
\usepackage{mathabx}
\usepackage[toc,page]{appendix}
\usepackage{hyperref}
\usepackage{cite}
\usepackage{bigints}
\usepackage[dvipsnames]{xcolor}

\usepackage[toc,page]{appendix}

\newtheorem{theorem}{Theorem}[section]
\newtheorem{proposition}[theorem]{Proposition}
\newtheorem{lemma}[theorem]{Lemma}
\newtheorem{corollary}[theorem]{Corollary}
\newtheorem{remark}{Remark}[section]

\newtheorem{assumption}[theorem]{Assumption}
\newtheorem{conjecture}[theorem]{Conjecture}

\theoremstyle{definition}
\newtheorem{definition}[theorem]{Definition}

\newcommand{\etal}{\textit{ et al. }}

\newcommand{\Tr}{{\mathrm{Tr}}}

\newcommand{\rank}{\mathrm{Rank}}

\newcommand{\D}{\mathcal D}

\newcommand{\vertiii}[1]{{\left\vert\kern-0.25ex\left\vert\kern-0.25ex\left\vert #1 
    \right\vert\kern-0.25ex\right\vert\kern-0.25ex\right\vert}}

\title{A rigorous justification of the Mittleman's approach to the Dirac--Fock model}

\begin{document}

\author{Long Meng}
\address{Long Meng:
CERMICS, \'Ecole des ponts ParisTech, 6 and 8 av. Pascal, 77455 Marne-la-Vall\'ee, France}
\email{long.meng@enpc.fr}

\keywords{Dirac--Fock model, electron-positron Hatree--Fock model, Mittleman's conjecture, ground state energy.}
\subjclass[2010]{35Q40, 46N50, 81V55}

\maketitle
\begin{abstract}
In this paper, we study the relationship between the Dirac--Fock model and the electron-positron Hartree--Fock model. We justify the Dirac--Fock model as a variational approximation of QED when the vacuum polarization is neglected and when the fine structure constant $\alpha$ is small and the velocity of light $c$ is large. As a byproduct, we also prove, when $\alpha$ is small or $c$ is large, the no-unfilled shells theory in the Dirac--Fock theory for atoms and molecules. The proof is based on some new properties of the Dirac--Fock model.
\end{abstract}

\section{Introduction}

Electrons in heavy atoms experience significant relativistic effects. It is widely believed that QED yields such a description. This paper addresses a conjecture due to Mittleman \cite{mittleman1981theory}: the Dirac--Fock model can be interpreted as a mean-filed approximation of Quantum Electrodynamics (QED) when the vacuum polarization is neglected and when the fine structure constant $\alpha$ is small and the velocity of light $c$ is large. More precisely, we prove that the error bound between the Dirac--Fock ground state energy and a max-min problem coming from the electron-positron Hartree--Fock model is of the size $\mathcal{O}(\frac{\alpha^2}{c^4})$.

In computational chemistry, the DF model firstly introduced in \cite{swirles1935relativistic} is frequently used. It is a variant of the Hartree--Fock model in which the kinetic energy operator $-\frac{1}{2}\Delta$ is replaced by the free Dirac operator $\D$. Even though in principle it is not physically meaningful, this approach gives better results that are in excellent agreement with experience data (see, e.g., \cite{desclaux1973relativistic,gorceix1987multiconfiguration}). Contrary to the models in QED, the DF functional is not bounded from below. The rigorous definition of ground state energy is thus delicate. Based on the critical point theory, rigorous existence results for solutions to the DF equations can be found in \cite{esteban1999solutions,paturel2000solutions}.  Then in \cite{esteban2001nonrelativistic}, a definition of the ground state energy based on the wavefunctions is proposed:
\begin{equation}\label{DF}
    E_{w,q}=\min_{\substack{ \Phi\in G_q(H^{1/2})\\ \Phi\textrm{ solution of } DF \textrm{ equations}}}\mathcal{E}(\gamma_\Phi).
\end{equation}
Here $\mathcal{E}$ is the DF functional defined by \eqref{eq:DFfunctional}, $\gamma_\Phi$ is the density matrix associated with $\Phi:=(u_1,\cdots,u_q)\in G_q$ defined by \eqref{eq:den-w}, and $\D_{\gamma_\Phi}$ is the DF operator defined by \eqref{eq:DF}. The space $G_q$ is the functional space presenting the wavefunctions of $q$ electrons and is a Grassmannian manifold defined by
\[
G_q(H^{1/2}):=\{G\textrm{ subspace of } H^{1/2}(\mathbb{R}^3;\mathbb{C}^4);\,\dim_\mathbb{C}(G)=q\}
\]
where $q$ is the number of electrons, and  the solutions of DF equation take the form
\begin{align*}
    \D_{\gamma_{\Phi}} u_j =\epsilon_j u_j,\qquad \epsilon_j\in (0,c^2).
\end{align*}

In addition, the deficiency that the DF Hamiltonian is not bounded from below also raises questions about its physical derivation: one would like to show that the DF model or its refined variants can be interpreted as an approximation of QED (c.f., \cite{mittleman1981theory,sucher1980foundations,chaix1989quantum} and the references therein). However, this theory leads to divergence problems: it is not easy to give meaning to the quantities (energy of the vacuum, charge density of the vacuum) appearing in QED. Steps in the direction of a rigorous justification are presently undertaken by considering the electron-positron Hartree--Fock model (ep-HF). This model is an approximation of QED with the vacuum polarization being neglected (see e.g., \cite[Section 4.5]{Variationalmethods}). Indeed, the vacuum polarization is of the size $\mathcal{O}(\frac{1}{c^3})$ which is small compared with the relativistic effect (of the size $\mathcal{O}(\frac{1}{c^2})$).

In the spirit of Mittleman \cite{mittleman1981theory}, when the vacuum polarization is neglected, the real physical ground state energy $e_q$ should be obtained by maximizing the ground state energy of the ep-HF model over all allowed one-particle electron subspace (see Formula \eqref{min-QED}). A version of the so-called \textit{Mittleman's conjecture} says that the corresponding ground state energy is equal to the ground state energy of the DF model. From a mathematical viewpoint, Mittleman's conjecture has been investigated. When the atomic shells are filled and the electron-electron interaction is weak \cite{barbaroux2005some,huber2007solutions} or in the case of hydrogen \cite{barbaroux2005remarks}, it works very well. However, all other cases are unknown.

The main result of this work (Theorem \ref{th:error bound}) is to answer this question in any case: when $\alpha$ is small and $c$ is large and when the vacuum polarisation is neglected, the DF ground state energy is an approximation of the physical ground state energy $e_q$ with a difference of the size $\mathcal{O}(\frac{\alpha^2}{c^4})$. In addition, the DF model is also a good approximation of QED even though the vacuum polarization is taken into account (see Remark \ref{rem:QED}).

Our strategy of proof involves some insight into the properties of the DF model. Recently S\'er\'e \cite{sere} redefined the DF ground state energy in the density matrix' framework (see Formula \eqref{eq:DFfunctional} and \eqref{eq:min}). The state of electrons in the DF theory is located in a subset $\Gamma_q^+$ of the density matrix in which any density matrix $\gamma$ satisfies $P^+_{\gamma}\gamma P^+_{\gamma}=\gamma$ with $P^+_{\gamma}=\mathbbm{1}_{(0,+\infty)}(\D_{\gamma})$. S\'er\'e's proof is based on a new retraction technique: he builds a retraction map $\theta$ that maps a subset of the density matrix into $\Gamma_q^+$. Thus, the difficulty of the nonlinear constraint $\gamma=P^+_{\gamma}\gamma P^+_{\gamma}$ is converted into the complexity of the structure of a new functional $\mathcal{E}(\theta(\cdot))$. Before going further, we review in Section \ref{sec:model} the basic definitions and results of the DF model, the ep-HF model, and Mittleman's definition of the ground state energy.

With the retraction map $\theta$ in hand, we show that for any pure electronic state $\gamma$ in ep-HF, the functional $\mathcal{E}(\theta(\gamma))$ is an approximation of the functional $\mathcal{E}(\gamma)$ when $\alpha$ is small or $c$ is large (i.e., Theorem \ref{main theorem}). Thus the DF model is expected to have the same properties as the ep-HF model when $\alpha$ is small or $c$ is large:
\begin{itemize}
    \item The functional $\mathcal{E}(\theta(\cdot))$ has a second-order expansion (i.e., Proposition \ref{approx});
    \item There are no unfilled shells in the DF theory (i.e., Theorem \ref{th:unfill}).
\end{itemize}
Therefore, any ground state energy of the DF model in the framework of density matrix (i.e., \eqref{eq:min}) is equivalent to the one in the framework of the wavefunction (i.e., \eqref{DF}) when $\alpha$ is small or $c$ is large (i.e., Corollary \ref{cor:equiv}).

Thanks to the equivalence of these two definitions of the DF ground state energy, the DF model is justified in Section \ref{sec:DF<ep-HF}: the DF ground state energy is an approximation of the ground state energy in QED when $\alpha$ is small and $c$ is large and when the vacuum polarisation is neglected

Finally, we give the technical proof of the error bound estimate between $\mathcal{E}(\cdot)$ and $\mathcal{E}(\theta(\cdot))$ (i.e., Theorem \ref{main theorem}) in Section \ref{sec:proof2}. In Appendix, we recall some basic inequalities used in \cite{sere} and prove the boundedness of the eigenfunctions of the DF operator and the DF type operator associated with the ep-HF model.

\section{Models, Mittleman's conjecture and main results}\label{sec:model}

For a particle of mass $m=1$, the free Dirac operator is defined by 
\[
\D=-ic\sum_{k=1}^3\alpha_k\partial_k+c^2\beta\]
with the velocity of light $c$ and $4\times4$ complex matrix $\alpha_1,\alpha_2,\alpha_3$ and $\beta$, whose standard forms are:
\[
\beta=\begin{pmatrix} 
\mathbbm{1}_2 & 0 \\
0 & -\mathbbm{1}_2 
\end{pmatrix},\,
\alpha=\begin{pmatrix}
0&\sigma_k\\
\sigma_k&0
\end{pmatrix},
\]
where $\mathbbm{1}_2$ is the $2\times 2$ identity matrix and the $\sigma_k$'s, for $k\in\{1,2,3\}$, are the well-known $2\times2$ Pauli matrix 
\[
\sigma_1=\begin{pmatrix}
0&1\\
1&0
\end{pmatrix},\,
\sigma_2=\begin{pmatrix}
0&-i\\
i&0
\end{pmatrix},\,
\sigma_3=\begin{pmatrix}
1&0\\
0&-1
\end{pmatrix}.
\]
The operator $\D$ acts on $4-$spinors; that is, on functions from $\mathbb{R}^3$ to $\mathbb{C}^4$. It is self-adjoint in $\mathcal{H}:=L^2(\mathbb{R}^3;\mathbb{C}^4)$, with domain $H^1(\mathbb{R}^3;
\mathbb{C}^4)$ and form domain $H^{1/2}(\mathbb{R}^3;\mathbb{C}^4)$ (denoted by $H^1$ and $H^{1/2}$ in the following, when there is no ambiguity). Its spectrum is $\sigma(\D)=(-\infty,-c^2]\cup[+c^2,+\infty)$. Following the notation in \cite{esteban1999solutions,paturel2000solutions}, we denote by $\Lambda^+$ and $\Lambda^-=\mathbbm{1}_{\mathcal{H}}-\Lambda^+$ respectively the two orthogonal projectors on $\mathcal{H}$ corresponding to the positive and negative eigenspaces of $\D$; that is
\[
\begin{cases}
\D\Lambda^+=\Lambda^+\D=\Lambda^+\sqrt{c^4-c^2\Delta}=\sqrt{c^4-c^2\Delta}\,\Lambda^+;\\
\D\Lambda^-=\Lambda^-\D=-\Lambda^-\sqrt{c^4-c^2\Delta}=-\sqrt{c^4-c^2\Delta}\,\Lambda^-.
\end{cases}
\]
Throughout the paper, $\mathcal{B}(W,Y)$ is the space of bounded linear maps from a Banach space $W$ to a Banach space $Y$, equipped with the norm 
\[
\|A\|_{\mathcal{B}(W,Y)}:=\sup_{u\in W,\,\|u\|_W=1}\|Au\|_Y.
\]
We denote $\mathcal{B}(W):=\mathcal{B}(W,W)$. The functional space $\sigma_1:=\sigma_1(\mathcal{H})$ is defined by
\begin{align*}
    \sigma_1:=\{\gamma\in\mathcal{B}(\mathcal{H});\Tr(|\gamma|)<+\infty\},
\end{align*}
endowed with the norm
\begin{align*}
    \|\gamma\|_{\sigma_1}:=\Tr(|\gamma|).
\end{align*}
We also define
\[
X^s:=\{\gamma\in\mathcal{B}({\mathcal{H}}); \gamma=\gamma^*, (1-\Delta)^{s/4}\gamma(1-\Delta)^{s/4}\in\sigma_1\},
\]
endowed  with the norm
\[
 \|\gamma\|_{X^s}:=\|(1-\Delta)^{s/4}\gamma(1-\Delta)^{s/4}\|_{\sigma_1}.
\]
In particular, we denote $X:=X^1$ and $Y:=X^2$.  For any $\gamma\in X$, we also introduce the following $c$-dependent norm:
\[
\|\gamma\|_{X_c}:=\||\D|^{1/2}\gamma|\D|^{1/2}\|_{\sigma_1}=\|(c^4-c^2\Delta)^{1/4}\gamma(c^4-c^2\Delta)^{1/4}\|_{\sigma_1}.
\]

For every density matrix $\gamma\in X$, there exists a complete set of eigenfunctions $(u_n)_{n\geq 1}$ of $\gamma$ in $\mathcal{H}$, corresponding to the non-decreasing sequence of eigenvalues $(\lambda_n)_{n\geq 1}$ (counted with their multiplicity) such that $\gamma$ can be rewritten as
\[
\gamma=\sum_{n\geq 1}\lambda_n |u_n \rangle\,\langle u_n|
\]
where $|u \rangle\,\langle u|$ denotes the projector onto the vector space spanned by the function $u$. The kernel $\gamma(x,y)$ of $\gamma$ reads as
\begin{align*}
    \gamma(x,y)=\sum_{n\geq 1}\lambda_n u_n(x)\otimes  u_n^*(y)
\end{align*}
The one-particle density associated with $\gamma$ is
\[
\rho_{\gamma}(x):=\Tr_{\mathbb{C}^4}[\gamma(x,x)]=\sum_{n\geq 1}\lambda_n |u_n(x)|^2,
\]
where the notation $\Tr_4$ stands for the trace of a $4\times4$ matrix. The density matrix $\gamma_{\Phi}$ corresponding to the wavefunctions of $q$ electrons $\Phi:=(u_1,\cdots,u_q)\in G_q(H^{1/2})$ can be written as
\begin{equation}\label{eq:den-w}
    \gamma_\Phi:=\sum_{n=1}^q|u_n \rangle\,\langle u_n|.
\end{equation}
For any $\gamma\in X$, the self-consistent DF operator is defined by
\begin{align}\label{eq:DF}
    \D_\gamma:=\D- V+\alpha W_\gamma
\end{align}
where for any $\psi\in H^{1/2}$,
\[
W_{\gamma}\psi(x)=(\rho_{\gamma}*W)\psi(x)-\int_{\mathbb{R}^3}W(x-y)\gamma(x,y)\psi(y)dy.
\]
Here $V$ is the attractive potential between the nuclei and the electrons, and $W$ is the repulsive potential between the electrons. We  consider the electrostatic case $W= \frac{1}{|x|}$ and $V=-\mu*\frac{1}{|x|}$ with a nonnegative nuclear charge distribution $\mu\in \mathcal{M}_+(\mathbb{R}^3)$ satisfying $\int_{\mathbb{R}^3} d\mu=Z$. 

The DF functional is 
\begin{align}\label{eq:DFfunctional}
    \mathcal{E}(\gamma):&=\Tr[(\D-V) \gamma]+\frac{\alpha}{2}\iint_{\mathbb{R}^3\times \mathbb{R}^3}\frac{\rho_{\gamma}(x)\rho_{\gamma}(y)-\Tr_{\mathbb{C}^4}(\gamma(x,y)\gamma(y,x))}{|x-y|}\;dxdy-c^2\Tr (\gamma)\notag\\
&=\Tr(\D_{\gamma} \gamma)-\frac{\alpha}{2}\Tr (W_{\gamma}\gamma)-c^2\Tr (\gamma).
\end{align}
\begin{remark}
Note that, for standard DF theory and in the system of units $\hbar=c=1$, the so-called fine-structure constant is $\alpha\approx \frac{1}{137}$, and $Z$ should be replaced by $\alpha Z$. Here we are rather interested in the case where $\alpha$ is small or $c$ is large.
\end{remark}

For future convenience, we denote $\alpha_c:=\frac{\alpha}{c}$ and $Z_c:=\frac{Z}{c}$.

\subsection{The Dirac--Fock ground state energy} 
Let $q$ be the number of electrons and let 
\[
\Gamma:=\{\gamma\in X; 0\leq \gamma\leq  \mathbbm{1}_{\mathcal{H}}\},\quad \Gamma_{q}=:\{\gamma\in \Gamma; \Tr(\gamma)\leq q\}.
\]
Let 
\[
P^+_{\gamma}=\mathbbm{1}_{(0,+\infty)}(\D_{\gamma}),\quad  P^-_{\gamma}=\mathbbm{1}_{(-\infty,0)}(\D_{\gamma}).
\]
In the DF theory, the relevant set of electronic states is defined by
\[
\Gamma_{q}^+:=\{\gamma\in\Gamma_q; P^+_\gamma\gamma P^+_\gamma=\gamma\}.
\]
According to \cite{sere}, the ground state energy of the DF model can be redefined by 
\begin{align}\label{eq:min}
    E_q:=\min_{\gamma\in \Gamma_{q}^+}\mathcal{E}(\gamma).
\end{align}
The existence of a ground state is guaranteed by the following.
\begin{theorem}[Existence of a ground state in the DF theory; Theorem 1.2 in \cite{sere}]\label{th:1} Let $\alpha q\leq Z$. Under Assumption \ref{ass:1} below, the minimum problem \eqref{eq:min} admits a minimizer $\gamma_* \in \Gamma_q^+$. In addition, $\Tr(\gamma_*)=q$, and any such minimizer can be written as
\begin{align}\label{eq:gamma}
    \gamma_*=\mathbbm{1}_{(0,\nu)}(\D_{\gamma_*})+\delta
\end{align}
with $
0< \delta\leq \mathbbm{1}_{\{\nu\}}(\D_{\gamma_*})$ for some $\nu\in (0,c^2]$. When $\alpha q<Z$, $\nu\in (0,c^2)$.
\end{theorem}
\begin{remark}
In fact $\delta=0$ is possible. But for future convenience, in this paper, we set $\delta>0$. If $\delta=0$, then there is a value $\nu'<\nu$ such that $\gamma_*=\mathbbm{1}_{(0,\nu']}(\D_{\gamma_*})$. 

Otherwise, $\mathbbm{1}_{(0,\nu']}(\D_{\gamma_*})<\gamma_*$ for any $\nu'<\nu$. Then, for any $0<\nu'<\nu$, there is a value $\nu''\in (\nu',\nu)$ such that $\nu''$ is an eigenvalue of the operator $\D_{\gamma_*}$. This implies that $\nu$ is an accumulation point of the spectrum and there are infinitely many positive eigenvalues of $\D_{\gamma_*}$ in $(0,\nu)$. Thus, $\Tr(\gamma_*)\geq \Tr[\mathbbm{1}_{(0,\nu)}(\D_{\gamma_*})]=+\infty$ which contradicts with the fact $\Tr(\gamma_*)\leq q$.
\end{remark}

\begin{assumption}\cite[Theorem 1.2 and Remark 1.3]{sere}\label{ass:1}
Let $\kappa(\alpha,c):=2(\alpha_c q+ Z_c)$. Assume that 
\begin{enumerate}
    \item  $\kappa(\alpha,c)< 1-\frac{\pi}{4}\alpha_c q$;
    \item  there exists $R>0$ such that : 
    \[
    (1-\kappa(\alpha,c)-\frac{\pi}{4}\alpha_cq)^{-1/2}q< R <\frac{2\sqrt{(1-\kappa(\alpha,c))\lambda_0(\alpha,c)}}{\pi\alpha_c}
    \]
\end{enumerate}
where $\lambda_0(\alpha,c):= (1-\max(\alpha_c q,Z_c))$.
\end{assumption}
The proof of Theorem \ref{th:1} is based on a retraction $\theta(\gamma)$ which is defined by 
\[
\theta(\gamma):=\lim_{n\rightarrow +\infty} T^n(\gamma)
\]
with
\[
T^n(\gamma)=T(T^{n-1}(\gamma)),\quad T(\gamma)=P^+_{\gamma}\gamma P^+_{\gamma}.
\]
The existence of the retraction $\theta$ is based on the following.
\begin{lemma}[Existence of the retraction; Proposition 2.1 and Proposition 2.9 in \cite{sere}]\label{lem:retra}
Assume that $\kappa(\alpha,c)<1$ and $a(\alpha,c):=\frac{\pi\alpha_c}{4\sqrt{(1-\kappa(\alpha,c))\lambda_0(\alpha,c)}}$. Given $R<\frac{1}{2a(\alpha,c)}$, let $A(\alpha,c):=\max(\frac{1}{1-2a(\alpha,c)R},\frac{2+a(\alpha,c)q}{2})$, and let
\[
\mathcal{U}_R:=\{\gamma\in\Gamma_q; \frac{1}{c}\|\gamma|\D|^{1/2}\|_{\sigma_1}+\frac{A(\alpha,c)}{c^2}\|T(\gamma)-\gamma\|_{X_c}< R\}.
\]
Then, $T$ maps $\mathcal{U}_R$ into $\mathcal{U}_R$, and for any $\gamma\in\mathcal{U}_R$ the sequence $(T^n(\gamma))_{n\geq 0}$ converges to a limit $\theta(\gamma)\in\Gamma_q^+$. Moreover for any $\gamma\in \mathcal{U}_R$,
\begin{equation}\label{eq:retra}
    \begin{aligned}
    \|T^{n+1}(\gamma)-T^{n}(\gamma)\|_{X_c}&\leq L(\alpha,c)\|T^n(\gamma)-T^{n-1}(\gamma)\|_{X_c},\\
    \|\theta(\gamma)-T^n(\gamma)\|_{X_c}&\leq\frac{L(\alpha,c)^n}{1-L(\alpha,c)}\|T(\gamma)-\gamma\|_{X_c}
\end{aligned}
\end{equation}
with $L(\alpha,c)=2a(\alpha,c)R$.
\end{lemma}
In particular, the property that $T(\mathcal{U}_R)\subset\mathcal{U}_R$ is shown in \cite[Proposition 2.9]{sere}. 

\begin{remark}\label{rem:R}
Let $q,R,Z$ be fixed. If $\alpha$ is small or $c$ is large, it is easy to see that $a(\alpha,c)\ll1$. Thus, the condition $R<\frac{1}{2a(\alpha,c)}$ is automatically fulfilled in this case.
\end{remark}
Therefore, one can define the DF energy:
\begin{definition}[DF energy]
Let $\kappa(\alpha,c),\;a(\alpha,c)$, $R$ and $\mathcal{U}_R$ be given as in Lemma \ref{lem:retra}. For any $\gamma\in\mathcal{U}_{R}$, the DF energy of $\gamma$ is defined by
\begin{align}\label{eq:def}
E(\gamma)=\mathcal{E}(\theta(\gamma)).
\end{align}
\end{definition}
According to \cite[Corollary 2.12]{sere}, any minimizer $\gamma_*$ of \eqref{eq:min} is located in $\mathcal{U}_R$ whenever Assumption \ref{ass:1} is satisfied under Assumption \ref{ass:1} on $\alpha,c$ and $R$. As a result, under Assumption \ref{ass:1},
\[
E_q=\min_{\gamma\in \mathcal{U}_{R}}E(\gamma).
\]

\subsection{Electron--positron Hartree--Fock theory}\label{th:HF}
The so-called electron-positron Hartree--Fock variational problem was introduced in the work of Bach\etal and Barbaroux\etal \cite{bach1998stability,bach1999stability,barbaroux2005hartree}: for any given Dirac sea $P^-_g:=1-P^+_{g}=\mathbbm{1}_{(-\infty,0)}(\D_{g})$ with $g\in X$, the set of electronic states in the ep-HF theory associated with the Dirac sea $P^-_g$ is defined by
\[
\Gamma^{(g)}_q:=\{\gamma\in X;\;-P^-_g\leq \gamma\leq P^+_g, \; P^+_g \gamma P^-_g=0,\; 0\leq \Tr (\gamma)\leq q\}.
\]
According to \cite[Lemma 3.7]{barbaroux2005hartree}, for any $g\in X$, the ep-HF ground state energy associated with the Dirac sea $P^-_g$ can be defined by
\begin{equation}\label{eq:min-HF}
    e^{(g)}_q:=\min_{\gamma\in \Gamma^{(g)}_q}\mathcal{E}(\gamma).
\end{equation}
The existence of the ground state in the ep-HF theory is proved in \cite[Theorem 3.9, Theorem 4.3 and Theorem 4.6]{barbaroux2005hartree}.
\begin{theorem}[Existence of the ground state in the ep-HF theory]\label{th:2}
Assume $g \in X$, $q\in \mathbb{N}$, $\alpha q\leq Z <\frac{\sqrt{3}}{2}c$ and 
\begin{align}\label{eq:ass-th2}
    \pi \alpha_c(1/4+\max\{\Tr(g), q\})+4\alpha_c\Tr(g)< \mu_{Z_c}.
\end{align}
Here $\mu_{Z_c}$ is a non-negative constant for any $Z_c\in [0,\frac{\sqrt{3}}{2})$ defined in Remark \ref{rem:value-v} below. Then the problem \eqref{eq:min-HF} has a minimizer $\gamma^{(g)}\in \Gamma_q^{(g)}$ satisfying $\Tr (\gamma^{(g)})=q$. In addition, there is no unfilled shell, i.e., for some $0< \nu <c^2$, 
\begin{equation}\label{eq:closed-shell}
    \gamma^{(g)}=\mathbbm{1}_{(0,\nu]}(P^+_{g}\D_{\gamma^{(g)}}P^+_{g}).
\end{equation}
\end{theorem}
\begin{remark}\label{rem:value-v}
In \cite[Lemma A.2]{barbaroux2005hartree} For any $a\in(-\frac{\sqrt{3}}{2},\frac{\sqrt{3}}{2})$, $\mu_{a}$ is the maximal value of $\mu$'s in $[0,C_a^2]$ such that 
\[
\mu+\frac{C^2_{a}a^2}{(C^2_{a}-\mu)}\leq 1
\]
with
\[
C_{a}:=\frac{1}{3}\left(-4|a|+\sqrt{9+4a^2}\right)>0.
\]
\end{remark}

\subsection{Mittleman's ground state energy and relativistic effect}
The set of Dirac seas in the ep-HF theory is defined by
\begin{align*}
    \mathcal{P}=\{P^-_{g}; g\in X\}.
\end{align*}
According to Mittleman \cite{mittleman1981theory}, if the vacuum polarization is neglected, the physical ground state energy is defined by
\begin{equation}\label{min-QED}
    e_q:=\sup_{P^-_g\in \mathcal{P}}e^{(g)}_q.
\end{equation}
One justification of the DF model relies on Mittleman's conjecture. A version is the following.
\begin{conjecture}[Mittleman's conjecture]\label{conj:1} For $\alpha$ small and $c$ large, $E_q=e_q$.
\end{conjecture}
Let $\sigma_i(D-V)$ be the $i$-th eigenvalue (counted with multiplicity) of the operator $D-V$. It is shown in \cite{barbaroux2005some,barbaroux2005remarks} that, Mittleman's conjecture \ref{conj:1} is true if $q=1$ or $\sigma_q(D-V)<\sigma_{q+1}(D-V)$ while it may be wrong for any other cases. Thus, we instead study the following weaker problem.
\begin{conjecture}[Weak Mittleman's conjecture]\label{conj:3}
The DF model is an approximation of the ep-HF model. More precisely, for $\alpha$ small and $c$ large, $e_q\leq E_q \leq e_q+o_{\alpha_c\to 0}(c^{-2})$.
\end{conjecture}
In relativistic quantum chemistry, the relativistic effects are of the order of $\frac{1}{c^2}$; that is the error bound between the relativistic models (e.g., Dirac--Fock) and the corresponding non-relativistic model (e.g., Hartree--Fock), see e.g., \cite{Kutzel1995direct1,Kutzel1995direct2}. Mathematically, in \cite{fournais2020scott} the error bound between the reduced Dirac--Fock model and the non-relativistic Thomas-Fermi model has been studied when $q$ is large enough and $\alpha q=\kappa$ is fixed. More precisely,
\begin{align*}
    E_q^{\rm rDF}=e_{\rm TF} \alpha^2 q^{\frac{7}{3}}+\frac{Z^2}{2}+C_{\rm rela}(c)+o_{q\to \infty}(1)
\end{align*}
with
\begin{align*}
   C_{\rm rela}(c):=\sum_{n\geq 1}\{\sigma_n(\D -\frac{Z}{|x|}-c^2)-\sigma_n(-\frac{\Delta}{2}-\frac{Z}{|x|})\}. 
\end{align*}
The term $e_{\rm TF}$ is a Thomas-Fermi ground energy, the term $\frac{Z^2}{2}$ is the Scott correction, and $C_{\rm rela}$ is the corresponding relativistic effect. It is easy to see that $C_{\rm rela}$ is of the order $\frac{1}{c^2}$.

The conjecture \ref{conj:3} expresses that the DF model indeed captures all relativistic effects of the order $\frac{1}{c^2}$ taken into account in the ground state energy $e_q$.

\subsection{Main results}
In this paper, we mainly focus on the non-relativistic regime (i.e., $c\gg1$) and the weak electron-electron interaction regime (i.e., $\alpha\ll 1$). In the non-relativistic limit (i.e., $c\to +\infty$), we have $\kappa(\alpha,c)\to 0$, $\lambda_0(\alpha,c)\to 1$ and, according to Remark \ref{rem:value-v}, $\mu_{Z_c}\to \mu_{0}=1$. Therefore, Assumption \ref{ass:1} and the condition \eqref{eq:ass-th2} ( for any $g\in \Gamma_q$) will be automatically satisfied as long as $c$ is large enough.

Analogously, in the weak electron--electron interaction limit, if $c>2Z$, then we have $\kappa(\alpha,c)\to 2Z_c<1$ and  $\lambda_0(\alpha,c)\to 1-Z_c>0$ when $\alpha\to 0$, and $\mu_{Z_c}>0$. Thus it is not difficult to see that Assumption \ref{ass:1} and the condition \eqref{eq:ass-th2} ( for any $g\in \Gamma_q$) will also be satisfied as long as $\alpha$ is sufficiently small.

As a consequence, we assume the following.
\begin{assumption}\label{ass:abzq}
Let $Z\in\mathbb{R}^+$ and $q\in\mathbb{N}^+$ be fixed. We assume that $\alpha,\,c\in\mathbb{R}^+$ are chosen such that $\alpha <\frac{Z}{q}$ and $c>2Z$ and one of the following conditions are satisfied:
\begin{enumerate}
    \item (Non-relativistic regime) $c$ is large enough;
    \item (Weak electron--electron interaction regime) $\alpha$ is small enough.
\end{enumerate}
\end{assumption}
\begin{remark}
The case $\alpha q=Z$ is out of reach as explained in Remark \ref{rem:unfill}.
\end{remark}
\begin{remark}\label{rem:ass}
Under Assumption \ref{ass:abzq}, it is easy to see that there is a constant $C>0$ independent of $\alpha$ and $c$ such that $C<1-\kappa(\alpha,c)\leq \lambda_0(\alpha,c)\leq 1$. Furthermore, when $R$ is fixed, then $R<\frac{1}{2a(\alpha,c)}$ and $L(\alpha,c)<1-C$ where $a(\alpha,c)$ and $L(\alpha,c)$ are given in Lemma \ref{lem:retra}.
\end{remark}
Recall that $E_{w,q}$ is the DF ground state energy based on the wavefunctions given in \eqref{DF} and $E_q$ is the DF ground state energy in the framework of the density matrix given in \eqref{eq:min}. Our first result concerns the relationship between these two definitions of the DF ground state energy.
 \begin{theorem}[There are no unfilled shells in the DF theory]\label{th:unfill}
Let $q\in \mathbb{N}^+$. Then under Assumption \ref{ass:abzq}, the no-unfilled shells property holds: any minimizer $\gamma_*$ of $E_q$ satisfies $\gamma_*=\mathbbm{1}_{(0,\nu]}(\D_{\gamma_*})$.
\end{theorem}
As a result,
\begin{corollary}\label{cor:equiv}
Under Assumption \ref{ass:abzq}, $E_q=E_{w,q}$.
\end{corollary}
The proof of Theorem \ref{th:unfill} and Corollary \ref{cor:equiv} are postponed to the end of Section \ref{sec:unfill}.
According to \cite[Formula (13)]{barbaroux2005some}, $e_q\leq E_{w,q}$ holds true for $\alpha$ sufficiently small and $c$ sufficiently large. Hence, 
\begin{align}\label{eq:low}
    e_q\leq E_q.
\end{align}
By definition, for any $g\in \mathcal{P}$, $e_q^{(g)}\leq e_q$. To prove Mittleman's conjecture, we investigate the relationship between $E_q$ and $e^{(g)}_q$ in Section \ref{sec:DF<ep-HF}.

The other main result  of this paper is the following.
\begin{theorem}[Weak Mittleman's conjecture]\label{th:error bound}
Let $Z\in\mathbb{R}^+$ and $q\in \mathbb{N}^+$. For $\alpha$ sufficiently small and $c$ sufficiently large, there exists a constant $C>0$ such that
\begin{equation}\label{eq:error bound}
    e_q\leq E_q\leq  e_q+C\frac{\alpha^2}{c^4},
\end{equation}
where $\gamma_*$ minimizes $E_q$ and $\gamma^{(\gamma_*)}$ minimizes $e^{(\gamma_*)}$.
\end{theorem}
The proof is provided in Section \ref{sec:DF<ep-HF}. As explained in the introduction, this error bound is much smaller than the relativistic effects and thus provides a justification of the DF model.

\begin{remark}
Actually, according to \cite[Definition 2]{barbaroux2005some}, any orthogonal projector in $\mathcal{H}$ is $\epsilon$-close to $\Lambda^+$ when $c$ is large enough. Thus, \eqref{eq:low} and Theorem \ref{th:error bound} remain true if we only assume that $c$ is large enough. 
\end{remark}

\begin{remark}[Justification of the DF model with vacuum polarization]\label{rem:QED}
In the full QED theory, typical QED effects such as vacuum polarization is of the size $\mathcal{O}(\frac{1}{c^3})$ (see, e.g., \cite[Ch. 5.5]{dyall2007introduction}). Then the ground state energy $e_q$ due to Mittleman and the DF ground state energy describe the full QED model up to an error bound of the size $\mathcal{O}(\frac{1}{c^3})$. 
\end{remark}

\section{Properties of the DF model}\label{sec:str}
Before studying Mittleman's conjecture, we need a better understanding of the properties of the DF model. The key ingredient in the proof of our weak Mittleman's conjecture is the following.
\begin{theorem}[Error bound between the DF functional and the DF energy]\label{main theorem}
Let $R,Z\in\mathbb{R}^+$ and $q\in\mathbb{N}^+$ be fixed. Let $\kappa(\alpha,z)<1$ and $L(\alpha,c)=2a(\alpha,c)R<1$ be given as in Lemma \ref{lem:retra}. Then for any $\gamma\in\mathcal{U}_R$ satisfying $P^+_{g}\gamma P^+_{g}=\gamma$ with $g\in \Gamma_q$, 
\begin{align}\label{E-E}
    |E(\gamma)-\mathcal{E}(\gamma)|\leq \frac{5\pi^2(6+\pi)( R+q)}{4(1-\kappa(\alpha,c))^4\lambda_0^{5/2}(\alpha,c)(1-L(\alpha,c))^2}\frac{\alpha^2}{c^2}\|g-\gamma\|_X^2.
\end{align}
If moreover $g,\gamma\in Y$,
\begin{align}\label{E-E'}
    |E(\gamma)-\mathcal{E}(\gamma)|\leq \frac{5(6+\pi)( R+q)}{(1-\kappa(\alpha,c))^4\lambda_0^{9/2}(\alpha,c)(1-L(\alpha,c))^2}\frac{\alpha^2}{c^4}\left(32\|g-\gamma\|_Y +c\|[W_{g-\gamma},\beta]\|_{\mathcal{B}(\mathcal{H})}\right)^2.
\end{align}
\end{theorem}
The proof is postponed until Section \ref{sec:proof2}. 

Before going further, let us roughly explain the implications of  this theorem~:
\begin{enumerate}
    \item For any pure electronic quantum state in the ep-HF model (i.e., any $\gamma\in\Gamma_q^{(g)}$ with $P^-_{g}\gamma P^-_g=0$), the DF energy of $\gamma$ is an approximation of the corresponding ep-HF energy. More precisely, if $\gamma\in\mathcal{U}_R$ for some $R>0$, then
    \[
    |E(\gamma)-\mathcal{E}(\gamma)|\leq C(\alpha,c)\alpha_c^2.
    \]
    \item The DF model is an approximation of the ep-HF model associated with the Dirac sea $P^-_{\gamma_*}$. This result will be proved in Section \ref{sec:DF<ep-HF}.
\end{enumerate}
Consequently, one can deduce that some properties of the ep-HF model can translate to the DF model as mentioned in the introduction.

\subsection{Second-order expansion for the DF energy}
It is easy to see that for any $\gamma$ and $h$ in $\Gamma_{q}^{(g)}$ and $t\in\mathbb{R}$ such that $\gamma+th\in \Gamma_q^{(g)}$, we have the following expansion for the ep-HF energy
\begin{align*}
    \mathcal{E}^{(g)}(\gamma+th)=\mathcal{E}^{(g)}(\gamma)+t\Tr [(\D_{\gamma}-c^2)h]+\frac{\alpha t^2}{2}\Tr[W_h h].
\end{align*}
Here $\mathcal{E}^{(g)}=\mathcal{E}|_{\Gamma_{q}^{(a)}}$ represents the energy of the electronic state in ep-HF theory. As a result of Theorem \ref{th:error bound}, we have the following expansion for the DF energy.
\begin{proposition}[Second-order expansion for the DF energy] \label{approx}
Let $R,Z\in\mathbb{R}^+$ and $q\in\mathbb{N}^+$ be fixed. Let $\kappa(\alpha,z)=2(\alpha_cq+Z_c)<1$ and $R<\frac{1}{2a(\alpha,c)}$. Given any $\gamma\in \mathcal{U}_R\cap \Gamma_q^+$ and any $h\in X$ such that $P^+_{\gamma}h P^+_{\gamma}=h$, then for any $t\in\mathbb{R}$ satisfying $\gamma+th\in \mathcal{U}_R$, we have
\begin{equation}\label{eq:sed-diff}
    \begin{aligned}
\MoveEqLeft    E(\gamma+th)=E(\gamma)+t\Tr [(\D_{\gamma}-c^2)h]+\frac{\alpha t^2}{2}\Tr[W_h h]+t^2\alpha_c^2  Error_{\gamma}(h,t),
\end{aligned}
\end{equation}
where 
\[
 |Error_{\gamma}(h,t)|\leq \frac{5\pi^2(6+\pi)( R+q)}{4(1-\kappa(\alpha,c))^4\lambda_0^{5/2}(\alpha,c)(1-L(\alpha,c))^2}\|h\|_X^2.
\]
\end{proposition}
\begin{proof}
Let $Error_{\gamma}(h,t):=\frac{1}{\alpha_c^2 t^2}[E(\gamma+th)-\mathcal{E}(\gamma+th)]$. Notice that $\gamma\in \mathcal{U}_R\cap \Gamma_q^+$ implies that $\theta(\gamma)=\gamma$ and $E(\gamma)=\mathcal{E}(\gamma)$. Then
\[
E(\gamma+th)=E(\gamma)+\mathcal{E}(\gamma+th)-\mathcal{E}(\gamma)+t^2\alpha_c^2 Error_\gamma(h,t).
\]
On the other hand, 
\[
\mathcal{E}(\gamma+th)-\mathcal{E}(\gamma)=t\Tr [(\D_{\gamma}-c^2)h]+\frac{\alpha t^2}{2}\Tr[W_h h].
\]
Hence \eqref{eq:sed-diff}. The boundedness of $Error_{\gamma}(h,t)$ follows from Theorem \ref{main theorem} directly as $\gamma+th=P^+_{\gamma}(\gamma+th)P^+_\gamma$.
\end{proof}

\begin{remark}
It is shown in \cite[Theorem 2.10]{sere} that the retraction $\theta$ is differentiable and its differential is bounded and uniformly continuous on $\mathcal{U}_R$. Moreover,
\begin{align}\label{eq:derivtheta}
    d\theta(\gamma)h=P^+_{\gamma}hP^+_{\gamma}+b_\gamma(h)+b_\gamma(h)^*,
\end{align}
where $b_\gamma(h)=P^+_{\gamma}b_\gamma(h)P^-_{\gamma}$. As a result, if $h=P^+_{\gamma}h P^+_{\gamma}$,
\[
E(\gamma+th)-E(\gamma)=t\Tr [(D_{\gamma}-c^2)h]+ t e_{\gamma}(h,t)
\]
where $e_{\gamma}(h,t)$ is equicontinuous with respect to $t\in [0,t_0]$ for some $t_0$ small enough. Compared to \cite{sere}, we can get the second differentiability of the DF energy $E(\cdot)$ w.r.t. $t$. 

Actually, according to \eqref{eq:derivtheta}, provided $h\in X$ such that $h=P^+_{\gamma}hP^+_{\gamma}$, one can prove 
\[
\|P^+_{\gamma}d^2\theta(\gamma)[h,h]P^+_{\gamma}\|_{X}+\|P^-_{\gamma}d^2\theta(\gamma)[h,h]P^-_{\gamma}\|_{X}<C(\alpha,c)\alpha_c^2\|h\|_X^2,
\]
from which we can also deduce Proposition \ref{approx} for  $t\in [0,t_0]$ for some $t_0$ small enough.
\end{remark}

\subsection{There are no unfilled shells in the Dirac--Fock theory}\label{sec:unfill}
In \cite[Theorem 4.6]{barbaroux2005hartree}, it has been proved that there are no unfilled shells in the ep-HF model. The proof is based on the second-order expansion of the ep-HF functional. Here we are going to use Proposition \ref{approx} to study the same property for the DF theory for $\alpha$ is small enough or $c$ is large enough: we turn to the
\begin{proof}[Proof of Theorem \ref{th:unfill}]
To prove the no-unfilled shell property, we only consider the non-relativistic regime: $\alpha q<Z$, and $c>2Z$ is large enough. The weak electron--electron interaction regime can be treated in the same manner and we only need to replace \cite[Theorem 3]{esteban2001nonrelativistic} by \cite[Lemma 2.1]{esteban1999solutions} in the proof.

For clarity, we add the superscript $c$ to the quantities depending on $c$. We argue by contradiction and assume that there is a sequence $c_n\to+\infty$ such that $\gamma_*^{c_n}\neq \mathbbm{1}_{(0,\nu^{c_n}]}(\D_{\gamma_*^{c_n}}^{c_n})$.

The idea of the proof is essentially the same as in \cite{PhysRevLett.72.2981, barbaroux2005hartree}: we shall find a sequence of density matrix $(h^{c_n})_n$ such that up to subsequences and for $t\in (0,t_0]$ with $t_0>0$ small enough,
\[
0\leq E^{c_n}(\gamma_*^{c_n}+th^{c_n})-E_q^{c_n}<0.
\]
However, compared to \cite{PhysRevLett.72.2981, barbaroux2005hartree}, the dependence on $c_n$ complicates the proof. As the velocity of light $c_n$ varies, the minimizers $\gamma_*^{c_n}$ will change as well. To reach the contradiction, we first need to check the existence of the sequences $(h^{c_n})_n$ and $(\gamma_*^{c_n}+th^{c_n})_n$ for any $t\in (0,t_0]$. To do so, we need the spectral analysis of $\D^{c_n}_{\gamma_*^{c_n}}$. Thus the proof is separated into 3 steps. In Step 1, we summarize the spectral properties of the DF operator $\D_{\gamma_*^{c_n}}^{c_n}$. In Step 2, we construct the sequence $(h^{c_n})_n$. To use Proposition \ref{approx}, we will find a constant $R>0$ such that under Assumption \ref{ass:abzq}, $\gamma_*^{c_n}\in \mathcal{U}_R$ and $\gamma_*^{c_n}+th^{c_n}\in \mathcal{U}_R$. Finally in Step 3, we reach the contradiction.

\medskip

\textbf{Step 1. Spectral analysis of $\D_{\gamma_*^{c_n}}^{c_n}$.} Here we use the following.
\begin{lemma}\cite[Lemma 3.4]{sere}\label{lem:sigma}
   Under Assumption \ref{ass:abzq}, there are constant $\tau<0$ independent of $c_n$ and integer $M>q$ independent of $c_n$ such that, for any $\gamma\in \Gamma_{q}$, the mean-filed operator $\D_{\gamma}^{c_n}$ has at least $q$ eigenvalues (counted with multiplicity) in the interval $[0,c^2-\tau]$ and has at most $M$ eigenvalues in $[0,1-\frac{\tau}{2}]$.
\end{lemma}
As a result of this lemma,
\begin{align}\label{eq:ran}
    \rank(\gamma_*^{c_n})\leq \rank(\mathbbm{1}_{(0,1-\frac{\tau}{2}]}(\D_{\gamma_*^{c_n}}^{c_n}))\leq M.
\end{align}
Then $\gamma_*^{c_n}$ can be written as
\[
\gamma_*^{c_n}=\sum_{k=1}^{M}\mu_k^{c}\left|\psi_k^{c}\right>\left<\psi_k^{c}\right|,\quad 0\leq \mu_k^{c_n}\leq 1,
\]
where $\psi_1^{c_n},\cdots, \psi_M^{c_n}$ are the first $M$ eigenfunctions of the operator $\D^{c_n}_{\gamma_*^{c_n}}$ with the eigenvalues $\nu_k^{c_n}$ such that $\nu_1^{c_n}\leq \nu_2^{c_n}\cdots \leq \nu_M^{c_n}$.

\medskip

\textbf{Step 2. Construction of  $h^{c_n}$.} Recall that $\gamma_*^{c_n}=\mathbbm{1}_{(0,\nu^{c_n})}(\D_{\gamma_*^{c_n}}^{c_n})+\delta^{c_n}$ with $0<\delta^{c_n}\leq \mathbbm{1}_{\{\nu^{c_n}\}}(\D_{\gamma_*^{c_n}}^{c_n})$. As $\gamma_*^{c_n}\neq \mathbbm{1}_{(0,\nu^{c_n}]}(\D_{\gamma_*^{c_n}}^{c_n})$, then $0< \delta^{c_n}<\mathbbm{1}_{\{\nu^{c_n}\}}(\D_{\gamma_*^{c_n}}^{c_n})$. Thus, it follows from \eqref{eq:ran} and the fact $\Tr (\delta^{c_n})\in\mathbb{N}$ that
\[
2\leq \rank(\mathbbm{1}_{\{\nu^{c_n}\}}(\D_{\gamma_*^{c_n}}^{c_n}))\leq M,
\]
and 
\[
1\leq \Tr (\delta^{c_n})\leq \Tr(\mathbbm{1}_{\{\nu^{c_n}\}}(\D_{\gamma_*^{c_n}}^{c_n}))-1=\rank(\mathbbm{1}_{\{\nu^{c_n}\}}(\D_{\gamma_*^{c_n}}^{c_n}))-1.
\]
Therefore, for some index $0\leq a^{c_n},b^{c_n}\leq M$ and $a^{c_n}\neq b^{c_n}$, there are two eigenvalues $\mu_{a^{c_n}}^{c_n}$ and $\mu_{b^{c_n}}^{c_n}$ of $\delta^{c_n}$ associated with the eigenfunctions $\psi_{a^{c_n}}^{c_n},\psi_{b^{c_n}}^{c_n}\in \ker(\D_{\gamma_*^{c_n}}^{c_n}-\nu^{c_n})$ such that
\begin{align}\label{eq:3.6}
    0\leq \mu_{a^{c_n}}^{c_n}\leq \frac{\Tr (\delta^{c_n})}{\rank(\mathbbm{1}_{\{\nu^{c_n}\}}(\D_{\gamma_*^{c_n}}^{c_n}))} \leq \frac{M-1}{M}\quad \textrm{and} \quad \frac{1}{M}\leq \frac{\Tr (\delta^{c_n})}{\rank(\mathbbm{1}_{\{\nu^{c_n}\}}(\D_{\gamma_*^{c_n}}^{c_n}))} \leq  \mu_{b^{c_n}}^{c_n}\leq 1.
\end{align}
Here we use the fact that, for any non-negative number $f_{1},\cdots,f_{J}$, there is always a constant $f_{j_1}$ and $f_{j_2}$ with $j_1,j_2\in \{1,\cdots,J\}$ such that
\begin{align*}
    f_{j_1}\leq \frac{1}{J}\sum_{j=1}^J f_j\leq f_{j_2}.
\end{align*}

Let $h^{c_n}=\left|\psi_{a^{c_n}}^{c_n}\right>\left<\psi^{c_n}_{a^{c_n}}\right|-\left|\psi_{b^{c_n}}^{c_n}\right>\left<\psi^{c_n}_{b^{c_n}}\right|$, and we take $t_0=\frac{1}{M}$. Then for $t\in (0,\frac{1}{M}]$, we have $\gamma_*^{c_n}+th^{c_n}\in \Gamma_q$. We are going to fix a constant $R>0$ independent of $c_n$ such that $\gamma_*^{c_n}\in \mathcal{U}_R^{c_n}$ and $\gamma_*^{c_n}+th^{c_n}\in\mathcal{U}_R^{c_n}$.

According to Lemma \ref{lem:unibound}, 
\[
\|\gamma_*^{c_n}\|_X\leq K^2q.
\]
It is easy to see that when $R>Kq$, we have $\gamma_*^{c_n}\in\mathcal{U}_R^{c_n}$ since $T_{c_n}(\gamma_*^{c_n})=\gamma_*^{c_n}$ and
\[
\frac{1}{c_n}\|\gamma_*^{c_n}|\D^{c_n}|^{1/2}\|_{\sigma_1}\leq \|\gamma_*^{c_n}(1-\Delta)^{1/4}\|_{\sigma_1}\leq q^{1/2}\|\gamma_*^{c_n}\|_X^{1/2}\leq qK<R.
\]

We are going to find the constant $R$ such that $\gamma_*^{c_n}+th^{c_n}\in \mathcal{U}_R^{c_n}$. For simplicity, let $\gamma_t^{c_n}:=\gamma_*^{c_n}+th^{c_n}$. As $0\leq \gamma_t^{c_n}\leq \mathbbm{1}_{(0,c_n^2)}(\D_{\gamma_*^{c_n}}^{c_n})$ and $\gamma_t^{c_n}\in \Gamma_q$, by Lemma \ref{lem:unibound} again,
\[
\frac{1}{c_n}\|\gamma_t^{c_n} |\D^{c_n}|^{1/2}\|_{\sigma_1}\leq K q.
\]
Then by Lemma \ref{lem:T-I} below and under Assumption \ref{ass:abzq}, there exists a constant $C$ such that
\begin{align*}
  \frac{A(\alpha,c_n)}{c^2}\|T_{c_n}(\gamma_t^{c_n})-\gamma_t^{c_n}\|_{X_{c_n}}
    &\leq C A(\alpha,c_n)R\frac{\alpha_{c_n}}{c_n^2} t\|h^{c_n}\|_X\\
    &\leq CqK^2 A(\alpha,c_n)R\frac{\alpha_{c_n}}{c_n^2}.
\end{align*}
We choose $R=2(1+K^2)q$. Then for $c_n$ large enough, according to Remark \ref{rem:ass}, $\frac{\pi}{4}\alpha_{c_n}\leq a(\alpha,c_n)\leq C\alpha_{c_n}$, and $A(\alpha,c_n)\leq 2$. Thus for $c_n$ sufficiently large, we infer
\[
\frac{A(\alpha,c_n)}{c^2}\|T_{c_n}(\gamma_t^{c_n})-\gamma_t^{c_n}\|_{X_{c_n}}\leq \frac{Cq \alpha}{c_n^3}R\leq \frac{R}{2}.
\]
Therefore,
\begin{align*}
    \MoveEqLeft \frac{1}{c_n}\|\gamma^{c_n}_t |\D^{c_n}|^{1/2}\|_{\sigma_1}+\frac{A(\alpha,c_n)}{c_n^2}\|T_{c_n}(\gamma^{c_n}_t)-\gamma_t^{c_n}\|_{X_{c_n}}\leq \frac{R}{2}+\frac{(1+K^2)q}{2}<R.
\end{align*}
Thus there is a constant $R>0$ independent of $c_n$ such that under Assumption \ref{ass:abzq}, $\gamma^{(\gamma_*)}\in \mathcal{U}_R^{c_n}$.

\medskip

\textbf{Step 3. Contradiction.} Denoting $\psi^{c_n}_{j}:=(f^n_{j,\alpha})_{q\leq \alpha\leq 4}$ for any $1\leq j\leq M$. Thanks to Proposition \ref{approx}, we have
    \begin{align}\label{eq:3.7}
    \MoveEqLeft 0\leq E^{c_n}(\gamma^{c_n}_*+th^{c_n})-E^{c_n}_q=E^{c_n}(\gamma^{c_n}_*+th^{c_n})-E^{c_n}(\gamma_*^{c_n})\notag\\
    &=(\nu^{c_n}-\nu^{c_n})t+\frac{\alpha t^2}{2}\Tr[W_{h^{c_n}} h^{c_n}]+t^2\alpha_{c_n}^2  Error^{c_n}_{\gamma_*^{c_n}}(h^{c_n})\notag\\
    &\leq C(\alpha,c) t^2(R+q)\alpha_{c_n}^2-\frac{\alpha t^2}{2}\int\limits_{\mathbb{R}^3\times \mathbb{R}^3}\sum_{\alpha,\beta} \frac{\left|\det\begin{pmatrix}f^n_{a^{c_n},\alpha}(x) & f^n_{b^{c_n},\beta}(x)\\ f^n_{a^{c_n},\beta}(y) & f^n_{b^{c_n},\beta}(y)\end{pmatrix}\right|^2}{|x-y|}dxdy\notag\\
    &\leq C\frac{\alpha^2t^2}{{c_n}^2}-\frac{\alpha t^2}{2}\min_{1\leq j<k\leq M}\int_{\mathbb{R}^3\times \mathbb{R}^3}\sum_{\alpha,\beta} \frac{\left|\det\begin{pmatrix}f^n_{j,\alpha}(x) & f^n_{k,\beta}(x)\\ f^n_{j,\beta}(y) & f^n_{k,\beta}(y)\end{pmatrix}\right|^2}{|x-y|}dxdy.
\end{align}

According to the properties of $\sigma_k^{c_n}$, we also have $\lim_{n\to+\infty}\nu^{c_n}-{c_n}^2\leq \lim_{n\to +\infty}\sigma_M^{c_n}-{c_n}^2<0$.  At this point, arguing as for \cite[Theorem 3]{esteban2001nonrelativistic}, up to subsequences, there is a sequence of functions $(\psi_k)_{1\leq k\leq M}$ such that $((\psi_k^{c_n})_{1\leq k\leq M})_n\to (\psi_k)_{1\leq k\leq M}$ in $H^{1}(\mathbb{R}^3;\mathbb{C}^4)$. Thanks to the positive definiteness of $\frac{1}{|x|}$, up to subsequences, for $c'$ large enough, there is a constant $C'$ such that
\begin{align}\label{eq:3.8}
    \inf_{{c_n}>c'}\min_{1\leq i<j\leq M}\int_{\mathbb{R}^3\times \mathbb{R}^3}\sum_{\alpha,\beta} \frac{\left|\det\begin{pmatrix}f^n_{j,\alpha}(x) & f^n_{k,\beta}(x)\\ f^n_{j,\beta}(y) & f^n_{k,\beta}(y)\end{pmatrix}\right|^2}{|x-y|}dxdy\geq C'>0.
\end{align}
Inserting \eqref{eq:3.8} into \eqref{eq:3.7}, we get that up to subsequences, for $c_n$ large enough,
\[
0\leq C\frac{\alpha^2t^2}{c_n^2}-C'\frac{\alpha t^2}{2}<0,
\]
reaching a contradiction. As a consequence, for any $c$ large enough (i.e., $c>c'$), the minimizer can be rewritten as 
    \[
    \gamma_*^{c}=\mathbbm{1}_{(0,\nu^{c}]}(\D^{c}_{\gamma_*^{c}}).
\]
This ends the proof.
\end{proof}
\begin{remark}\label{rem:unfill}
Unfortunately, due to Lemma \ref{lem:sigma}, we can not reach the case $\alpha q=Z$. According to Theorem \ref{th:1}, when $\alpha q=Z$, the case $\nu=c^2$ may occur. In this case, if Theorem \ref{th:unfill} still holds (i.e., $\gamma_*^c=\mathbbm{1}_{(0,c^2]}(\D_{\gamma_*^c}^c)$), then $\Tr (\gamma_*^c)=+\infty$ which is impossible. However, once we have shown that $\nu<c^2$ strictly when $\alpha q=Z$, the no-unfilled shell property can be reached.
\end{remark}

We now use the no-unfilled shells property to prove the following.
\begin{proof}[Proof of Corollary \ref{cor:equiv}]
Let $\gamma_*$ be a minimizer of $E_{q}$. Actually, $\gamma_*=\mathbbm{1}_{(0,\nu]}(\D_{\gamma_*})$ is a projector and $\Tr(\gamma_*)=q$. Then the minimizer $\gamma_*$ can be rewritten as 
\[
\gamma_*:=\sum_{n=1}^q\left|u_n\right>\left<u_n\right|.
\]
Then $\Phi_*=(u_1,\cdots,u_n)\in G_q(H^{1/2})$ and $\Phi_*$ solves the DF equations. Consequently, $E_q=\mathcal{E}(\gamma_*)\leq E_{w,q}$. 

On the other hand, let $\Phi_*:=(u_1,\cdots,u_q)$ be a minimizer of $E_{w,q}$. Then $\Phi_*$ is a solution of the DF equation. Thus $P^+_{\gamma_{\Phi_*}}\gamma_{\Phi_*}P^+_{\gamma_{\Phi_*}}=\gamma_{\Phi_*}$. Hence $\gamma_{\Phi_*}\in \Gamma_q^+$, and $E_{w,q}=\mathcal{E}(\gamma_{\Phi_*})\leq E_q$. As a result, we know that $E_q=E_{w,q}$ under Assumption \ref{ass:abzq}.
\end{proof}

\section{DF model is an approximation of the ep-HF model}\label{sec:DF<ep-HF}
We are in the position to prove Theorem \ref{th:error bound}. It is an immediate consequence of \eqref{eq:low} and the following proposition.

\begin{proposition}[DF model is an approximation of the ep-HF model]\label{prop:approx}
Let $\gamma_*\in \Gamma_q^+$ be a minimizer of $E_q$ and $\gamma^{(\gamma_*)}$ be a minimizer of $e^{(\gamma_*)}$. Under Assumption \ref{ass:abzq}, there exists a constant $C>0$ such that
\begin{align}\label{eq:4.1}
    e_q^{(\gamma_*)}\leq E_q\leq e^{(\gamma_*)}+C\frac{\alpha^2}{c^4}.
\end{align}
\end{proposition}
\begin{proof}
By the definition of $E_q$ and $e_q^{(\gamma_*)}$ and as $\gamma_*\in \Gamma_q^{(\gamma_*)}$, it is easy to see that for any $\gamma^{(\gamma_*)}\in \mathcal{S}(\gamma_*)$,
\[
e_q^{(\gamma_*)}=\mathcal{E}(\gamma^{(\gamma_*)})\leq \mathcal{E}(\gamma_*)=E_q.
\]
Thus we just need to prove the second inequality in \eqref{eq:4.1}. 

We are going to find a constant $R>0$ independent of $\alpha$ and $c$ such that for $\alpha$ sufficiently small and $c$ sufficiently large, $\gamma^{(\gamma_*)}\in\mathcal{U}_R$. Once $R>0$ is chosen, the second inequality follows from Theorem \ref{main theorem}: as $P^+_{\gamma_*}\gamma^{(\gamma_*)}P^+_{\gamma_*}=\gamma^{(\gamma_*)}$,
\[
|E(\gamma^{(\gamma_*)})-\mathcal{E}(\gamma^{(g)}) |\leq C\frac{\alpha^2}{c^4}\left(\|\gamma_*-\gamma^{(\gamma_*)}\|_Y + c\|[W_{\gamma_*-\gamma^{(\gamma_*)}},\beta]\|_{\mathcal{B}(\mathcal{H})}\right)^2.
\]
According to Lemma \ref{lem:unibound} and Lemma \ref{lem:unibound'}, we have
\begin{align*}
    \|\gamma^{(\gamma_*)}-\gamma_*\|_Y\leq \|\gamma^{(\gamma_*)}\|_Y+\|\gamma_*\|_Y\leq 2K^2q.
\end{align*}
On the other hand, notice that for any function $u\in\mathcal{H}$,
\begin{align}
    [W_{\gamma_*-\gamma^{(\gamma_*)}},\beta]u =-\int_{\mathbb{R}^3}\frac{[(\gamma_*-\gamma^{(\gamma_*)})(x,y),\beta]u(y)}{|x-y|}dy
\end{align}
We rewrite $\gamma_*$ as
\[
\gamma_*=\sum_{j=1}^{+\infty} \mu_j\left|\psi_j\right>\left<\psi_j\right|
\]
where $\mu_j\geq 0$, $\sum_{j=1}^{+\infty}\mu_j=q$ and $\psi_j$ is the normalized eigenfunctions of $\D_{\gamma_*}$ with eigenvalues $0\leq\lambda_j\leq c^2$. We split $\psi_j$ in  blocks of upper and lower components:
\begin{align}
    \psi_j=\begin{pmatrix}
        \psi_j^{(1)}\\\psi_j^{(2)}
    \end{pmatrix},\quad\textrm{with}\quad \psi_j^{(1)},\;\psi_j^{(2)}:\;\mathbb{R}^3\to \mathbb{C}^2.
\end{align}
Then,
\begin{align*}
    \gamma_*(x,y)=\sum_{j=1}^{+\infty} \mu_j \begin{pmatrix}
        \psi_j^{(1)}(x)\otimes \psi_j^{(1),*}(y)& \psi_j^{(1)}(x)\otimes \psi_j^{(2),*}(y)\\ \psi_j^{(2)}(x)\otimes \psi_j^{(1),*}(y) & \psi_j^{(2)}(x)\otimes \psi_j^{(2),*}(y)
    \end{pmatrix}.
\end{align*}
Notice that
\begin{align*}
    [\gamma_*(x,y),\beta]=2\sum_{j=1}^{+\infty} \mu_j \begin{pmatrix}
        0& -\psi_j^{(1)}(x)\otimes \psi_j^{(2),*}(y)\\ \psi_j^{(2)}(x)\otimes \psi_j^{(1),*}(y) & 0
    \end{pmatrix}.
\end{align*}
According to \eqref{eq:prop:eigen1}, it is easy to see that
\begin{align*}
    c\|[W_{\gamma_*},\beta]\|_{\mathcal{B}(\mathcal{H})}\leq 8c\sum_{j=1}^\infty \mu_j \|\psi_j^{(1)}\|_{H^1}\|\psi^{(2)}_j\|_{\mathcal{H}}\leq 8KK'q.
\end{align*}
Analogously, according to \eqref{eq:prop:eigen2}, we also have
\begin{align*}
    c\|[W_{\gamma^{(\gamma_*)}},\beta]\|_{\mathcal{B}(\mathcal{H})}\leq 8KK'q.
\end{align*}
Consequently,
\begin{align*}
\MoveEqLeft E_q\leq E(\gamma^{(\gamma_*)})-\mathcal{E}(\gamma^{(\gamma_*)}+\mathcal{E}(\gamma^{(\gamma_*)})\leq  e_q^{(\gamma_*)}+|E(\gamma^{(\gamma_*)})-\mathcal{E}(\gamma^{(\gamma_*)})|\leq e_q^{(\gamma_*)} +C\frac{\alpha^2}{c^4}.
\end{align*}
Hence \eqref{eq:4.1}.

The method to find the constant $R>0$ is the same as in the proof of Theorem \ref{th:unfill}. By using Lemma \ref{lem:unibound'}, it is not difficult to see that
\[
\frac{1}{c}\|\gamma^{(\gamma_*)} |\D|^{1/2}\|_{\sigma_1}\leq \|\gamma(1-\Delta)^{1/4}\|_{\sigma_1}\leq q^{1/2}\|\gamma\|_X^{1/2}\leq K q.
\]
By Lemma \ref{lem:T-I} below, we have
\begin{align*}
    \MoveEqLeft\frac{A(\alpha,c)}{c^2}\|T(\gamma^{(\gamma_*)})-\gamma^{(\gamma_*)}\|_{X_c}\leq CK^2 R\frac{\alpha}{c^3}q.
\end{align*}
We choose $R=2(1+K^2)q$. Thus under Assumption \ref{ass:abzq}, we have
\[
\frac{A(\alpha,c)}{c^2}\|T(\gamma^{(\gamma_*)})-\gamma^{(\gamma_*)}\|_{X_c}\leq \frac{Cq \alpha}{c^3}R\leq \frac{R}{2}.
\]
Therefore,
\begin{align*}
    \MoveEqLeft \frac{1}{c}\|\gamma |\D|^{1/2}\|_{\sigma_1}+\frac{A(\alpha,c)}{c^2}\|T(\gamma^{(\gamma_*)})-\gamma^{(\gamma_*)}\|_{X_c}\leq \frac{R}{2}+\frac{(1+K^2)q}{2}<R,
\end{align*}
which means there is a constant $R>0$ such that for $\alpha$ sufficiently small and $c$ sufficiently large, $\gamma^{(g)}\in \mathcal{U}_R$.
\end{proof}

We turn now to the
\begin{proof}[Proof of Theorem \ref{th:error bound}]
Notice that $e_q^{(\gamma_*)}\leq e_q$. Then thanks to Proposition \ref{prop:approx}, we conclude that
\begin{align*}
    E_q\leq e_q+C\frac{\alpha^2}{c^4}.
\end{align*}
This gives the second inequality of \eqref{eq:error bound}. Combining with \eqref{eq:low}, we get the theorem.
\end{proof}

\section{Error bound of the DF functional and energy}\label{sec:proof2}
 This section is devoted to the proof of Theorem \ref{main theorem} which will be separated into two parts: estimate on \eqref{E-E} and estimate on \eqref{E-E'}.  For future convenience, we denote $L:=L(\alpha,c)$, $\kappa:=\kappa(\alpha,c)$ and $\lambda_0:=\lambda_0(\alpha,c)$ when there is no ambiguity.

 We first consider the error bound for any $\gamma\in \mathcal{U}_R$.
 \begin{lemma}\label{lem:error-bound}
 Let $R,Z\in\mathbb{R}^+$ and $q\in\mathbb{N}^+$ be fixed. Assume that $\kappa<1$ and $L<1$ as in Lemma \ref{lem:retra}. Let $C_{\kappa,L}:=\frac{5\pi^2}{4(1-\kappa)^2\lambda_0^{3/2}(1-L)^2}$. Then for any $\gamma\in\mathcal{U}_R$, 
\begin{align}
    |E(\gamma)-\mathcal{E}(\gamma)|\leq C_{\kappa,L}(3R\alpha_c+3q\alpha_c+1)\frac{\alpha_c}{c^2}\|T(\gamma)-\gamma\|_{X_c}^2+3\|P^-_{\gamma}\gamma P^-_{\gamma}\|_{X_c}
\end{align}
 \end{lemma}
 This is an immediate result of the following.
 \begin{lemma}\label{lem:t-T}
Let $C_{\kappa,L}$ be given in Lemma \ref{lem:error-bound}. With the same assumptions as in Lemma \ref{lem:error-bound}, for any $\gamma\in \mathcal{U}_R$ we have
\begin{align}\label{theta-T+}
    \|P^{+}_{\gamma}(\theta(\gamma)-T(\gamma))P^{+}_{\gamma}\|_{X_c}\leq C_{\kappa,L}\frac{R\alpha_c^2}{c^2}\|T(\gamma)-\gamma\|_{X_c}^2
\end{align}
and
\begin{align}\label{theta-T-} \|P^-_{\gamma}\theta(\gamma)P^-_{\gamma}\|_{X_c}\leq C_{\kappa,L}\frac{q\alpha_c^2}{c^2}\|T(\gamma)-\gamma\|_{X_c}^2.
\end{align}
\end{lemma}
We first use it to prove  Lemma \ref{lem:error-bound}.
\begin{proof}[Proof of Lemma \ref{lem:error-bound}]
Notice that
\begin{equation}\label{eq:E-Ecal}
    \begin{aligned}
   \MoveEqLeft E(\gamma)-\mathcal{E}(\gamma)=\Tr [\D_{\gamma}(\theta(\gamma)-\gamma)]+\frac{\alpha}{2} \Tr[W_{\theta(\gamma)-\gamma}(\theta(\gamma)-\gamma)]-c^2\Tr(\theta(\gamma)- \gamma).
\end{aligned}
\end{equation}
To end the proof, it suffices to calculate each term on the right-hand side separately.

\textbf{Estimate on $\Tr [\D_{\gamma}(\theta(\gamma)-\gamma)]$.} We consider the first term on the right-hand side of \eqref{eq:E-Ecal}. As $T(\gamma)=P^+_{\gamma} \gamma P^+_{\gamma}$, we have
\begin{align*}
    \Tr [\D_{\gamma}(\theta(\gamma)-\gamma)]&= \Tr [(P^+_{\gamma}+P^-_{\gamma})\D_{\gamma}(\theta(\gamma)-\gamma)(P^+_{\gamma}+P^-_{\gamma})]\\
   &=\Tr [|\D_{\gamma}|P^+_{\gamma}(\theta(\gamma)-T(\gamma))P^+_{\gamma}]-\Tr [|\D_{\gamma}|P^-_{\gamma}(\theta(\gamma)-\gamma)P^-_{\gamma}].
\end{align*}
Then by \eqref{eq:D-D} and the fact that $0\leq \kappa\leq 1$, we have
\begin{align*}
   \left| \Tr [|\D_{\gamma}|P^+_{\gamma}(\theta(\gamma)-\gamma)P^+_{\gamma}]\right|
   &\leq \||\D_\gamma|^{1/2}P^+_{\gamma}(\theta(\gamma)-T(\gamma))P^+_{\gamma}\D_\gamma|^{1/2}\|_{\sigma_1}\\
   &\leq 2\|P^+_{\gamma}(\theta(\gamma)-\gamma) P^+_{\gamma}\|_{X_c}\leq 2C_{\kappa,L}\frac{R\alpha_c^2}{c^2}\|T(\gamma)-\gamma\|_{X_c}^2.
\end{align*}
On the other hand, by \eqref{eq:D-D}, \eqref{eq:P-gamma-P}, we have
\begin{align*}
   \left| \Tr[|\D_{\gamma}|P^-_{\gamma}(\theta(\gamma)-\gamma)P^-_{\gamma}]\right|&\leq \||\D_\gamma|^{1/2}P^-_{\gamma}\theta(\gamma)P^-_{\gamma}\D_\gamma|^{1/2}\|_{\sigma_1}+\||\D_\gamma|^{1/2}P^-_{\gamma}\gamma P^-_{\gamma}\D_\gamma|^{1/2}\|_{\sigma_1}\\
   &\leq 2\left(\|P^-_{\gamma}\theta(\gamma)P^-_{\gamma}\|_{X_c}+\|P^-_{\gamma}\gamma P^-_{\gamma}\|_{X_c}\right)\\
   &\leq 2\left(C_{\kappa,L}\frac{q\alpha_c^2}{c^2}\|T(\gamma)-\gamma\|_{X_c}^2+\|P^-_{\gamma}\gamma P^-_{\gamma}\|_{X_c}\right).
\end{align*}
Then we conclude that
\begin{align}\label{eq:E-Ecal1}
    \left|\Tr [\D_{\gamma}(\theta(\gamma)-\gamma)]\right|\leq 2C_{\kappa,L}(R+q)\frac{\alpha_c^2}{c^2}\|T(\gamma)-\gamma\|_{X_c}^2+2\|P^-_{\gamma}\gamma P^-_{\gamma}\|_{X_c}.
\end{align}

\textbf{Estimate on $c^2\Tr(\theta(\gamma)- \gamma)$.} The term $c^2\Tr(\theta(\gamma)- \gamma)$ can be treated analogously. Actually,
\begin{align*}
c^2\left|\Tr [(\theta(\gamma)-\gamma)\right|&\leq  c^2\left|\Tr [|P^+_{\gamma}(\theta(\gamma)-\gamma)P^+_{\gamma}]+\Tr [P^-_{\gamma}(\theta(\gamma)-\gamma)P^-_{\gamma}]\right|\\
    &\leq \|P^+_{\gamma}(\theta(\gamma)-\gamma)P^+_{\gamma}\|_{X_c}+\|P^-_{\gamma}\theta(\gamma)P^-_{\gamma}\|_{X_c}+\|P^-_{\gamma}\gamma P^-_{\gamma}\|_{X_c}.
\end{align*}
Then proceeding as for the term $\Tr [\D_{\gamma}(\theta(\gamma)-\gamma)]$, we obtain
\begin{align}\label{eq:E-Ecal2}
    c^2\left|\Tr [(\theta(\gamma)-\gamma)\right|\leq C_{\kappa,L}(R+q)\frac{\alpha_c^2}{c^2}\|T(\gamma)-\gamma\|_{X_c}^2+\|P^-_{\gamma}\gamma P^-_{\gamma}\|_{X_c}.
\end{align}

\textbf{Estimate on $\alpha\Tr[W_{\theta(\gamma)-\gamma}(\theta(\gamma)-\gamma)]$.}  Using \eqref{eq:retra} and \eqref{eq:5.1-1}, we deduce
\begin{align}\label{eq:E-Ecal3}
  \alpha\left|\Tr[W_{\theta(\gamma)-\gamma}(\theta(\gamma)-\gamma)]\right|&\leq\frac{\pi}{2c^2}\alpha_c\|\theta(\gamma)-\gamma\|_{X_c}^2\leq \frac{\pi}{2(1-L)^2}\frac{\alpha_c}{c^2}\|T(\gamma)-\gamma\|_{X_c}^2\leq C_{\kappa,L}\frac{\alpha_c}{c^2}\|T(\gamma)-\gamma\|_{X_c}^2.
\end{align}

\textbf{Conclusion.} Gathering together \eqref{eq:E-Ecal}-\eqref{eq:E-Ecal3}, we conclude that
\begin{align*}
    \left|E(\gamma)-\mathcal{E}(\gamma)\right|\leq C_{\kappa,L}(3R\alpha_c+3q\alpha_c+1)\frac{\alpha_c}{c^2}\|T(\gamma)-\gamma\|_{X_c}^2+3\|P^-_{\gamma}\gamma P^-_{\gamma}\|_{X_c}.
\end{align*}
This gives \eqref{E-E}.
\end{proof}
 
\medskip

It remains to prove Lemma \ref{lem:t-T}. Before going further, we need the following.

\begin{proposition}
Let $\gamma,\gamma'\in \Gamma_q$ and $h\in X$. Then
\begin{align}\label{eq:dP}
    P_{\gamma}^+ (dP^+_{\gamma}h) P^+_{\gamma}=0,\quad P_{\gamma}^- (dP^+_{\gamma}h) P^-_{\gamma}=0
\end{align}
where $dP^+_{\gamma}h$ is the Gateaux derivative which is defined by 
\[
dP^+_{\gamma}h:=\lim_{t\to 0}\frac{P^+_{\gamma+th}-P^+_{\gamma}}{t}.
\]
Besides, we have
\begin{align}\label{eq:P-P-dP}
    \||\D|^{1/2}[P^+_{\gamma}-P^+_{\gamma'}-dP^+_{\gamma'}(\gamma-\gamma')]\|_{\mathcal{B}(\mathcal{H})}\leq \frac{\pi^2\alpha_c^2}{8c^3}(1-\kappa)^{-1/2}\lambda_0^{-3/2}\|\gamma-\gamma'\|_{X_c}^2.
\end{align}
\end{proposition}
\begin{proof}
As $P^+_{\gamma+th}=(P^+_{\gamma+th})^2$, for any $h\in X$ we have
\[
dP^+_{\gamma}h=P^+_{\gamma}(dP^+_{\gamma}h)+(dP^+_{\gamma}h) P^+_{\gamma}.
\]
Thus,
\[
P^-_{\gamma}(dP^+_{\gamma}h)P^-_{\gamma}=0,
\]
and
\[
P^+_{\gamma}(dP^+_{\gamma}h)P^+_{\gamma}=2P^+_{\gamma}(dP^+_{\gamma}h)P^+_{\gamma},
\]
hence \eqref{eq:dP}. 

We turn now to prove \eqref{eq:P-P-dP}. 
We recall that
\[
P^+_{\gamma}-P^+_{\gamma'}=\frac{\alpha}{2\pi}\int_{-\infty}^{+\infty}(\D_{\gamma}-iz)^{-1}W_{\gamma'-\gamma}(\D_{\gamma'}-iz)^{-1}dz
\]
and
\[
dP^+_{\gamma}(\gamma-\gamma')=\frac{\alpha}{2\pi}\int_{-\infty}^{+\infty}(\D_{\gamma}-iz)^{-1}W_{\gamma'-\gamma}(\D_{\gamma}-iz)^{-1}dz.
\]
Then,
\[
P^+_{\gamma}-P^+_{\gamma'}-dP^+_{\gamma}(\gamma-\gamma')=-\frac{\alpha^2}{2\pi}\int_{-\infty}^{+\infty}(\D_{\gamma}-iz)^{-1}W_{\gamma'-\gamma}(\D_{\gamma'}-iz)^{-1}W_{\gamma'-\gamma}(\D_{\gamma}-iz)^{-1}dz.
\]
From Lemma \ref{lem:ope} for any $\phi,\psi\in \mathcal{H}$, we deduce
\begin{align}\label{eq:P-P-dP2}
    \MoveEqLeft \left(\phi,|\D|^{1/2}[P^+_{\gamma}-P^+_{\gamma'}-dP^+_{\gamma}(\gamma-\gamma')]\psi\right)\notag\\
    &\leq \frac{\alpha^2}{2\pi}\|W_{\gamma'-\gamma}\|^2_{\mathcal{B}(\mathcal{H})}\||\D_{\gamma'}|^{-1}\|_{\mathcal{B}(\mathcal{H})}\!\!\left(\int_{-\infty}^{+\infty}\|(\D_{\gamma}-iz)^{-1}|\D|^{1/2}\phi\|_{\mathcal{H}}^2 dz\right)^{1/2}\!\!\!\!\left(\int_{-\infty}^{+\infty}\|(\D_{\gamma}-iz)^{-1}\psi\|_{\mathcal{H}}^2 dz\right)^{1/2}\notag\\
    &\leq \frac{\pi^2\alpha_c^2}{8c^2}\lambda_0^{-1}\|\gamma-\gamma'\|_{X_c}^2\||\D_{\gamma}|^{-1/2}|\D|^{1/2}\phi\|_{\mathcal{H}}\||\D_{\gamma}|^{-1/2}\psi\|_{\mathcal{H}}\notag\\
    &\leq \frac{\pi^2\alpha_c^2}{8c^3}(1-\kappa)^{-1/2}\lambda_0^{-3/2}\|\gamma-\gamma'\|_{X_c}^2\|\phi\|_{\mathcal{H}}\|\psi\|_{\mathcal{H}}.
\end{align}
This proves \eqref{eq:P-P-dP}.
\end{proof}

We now turn to the
\begin{proof}[Proof of Lemma \ref{lem:t-T}]
We first prove \eqref{theta-T+}. Indeed, it suffices to prove
\begin{equation}\label{eq:inter}
    \|P^{+}_\gamma(T^n(\gamma)-T^{n-1}(\gamma))P^{+}_\gamma\|_{X_c}\leq C_{\kappa,L}(1-L)\frac{R\alpha_c^2}{c^2}\|T(\gamma)-\gamma\|_{X_c}\|T^{n-1}(\gamma)-T^{n-2}(\gamma)\|_{X_c}.
\end{equation}
Then thanks to \eqref{eq:retra},
\begin{align*} \MoveEqLeft\|P^{+}_\gamma(\theta(\gamma)-T(\gamma))P^{+}_\gamma\|_{X_c}\leq \sum_{n=2}^{+\infty}\|P^{+}_\gamma(T^n(\gamma)-T^{n-1}(\gamma))P^{+}_\gamma\|_{X_c}\leq C_{\kappa,L}\frac{R\alpha_c^2}{c^2}\|T(\gamma)-\gamma\|_{X_c}^2.
\end{align*}

\medskip

We turn to prove \eqref{eq:inter}. Let $\gamma_n=T^n(\gamma)$ and $\gamma_0=\gamma$. Then for $n\geq 2$, $\gamma_n=P^+_{\gamma_{n-1}}\gamma_{n-1}P^+_{\gamma_{n-1}}$ and $\gamma_{n-1}=P^+_{\gamma_{n-2}}\gamma_{n-1}P^+_{\gamma_{n-2}}$. Hence, for $n\geq 2$
\begin{equation}\label{eq:decom}
    \begin{aligned}
    P^+_\gamma(\gamma_n-\gamma_{n-1})P^+_\gamma&=P^+_\gamma(P^+_{\gamma_{n-1}}-P^+_{\gamma_{n-2}})P^+_{\gamma_{n-2}}\gamma_{n-1}P^+_{\gamma_{n-1}}P^+_\gamma\\
    &\quad+P^+_\gamma\gamma_{n-1}P^+_{\gamma_{n-2}}(P^+_{\gamma_{n-1}}-P^+_{\gamma_{n-2}})P^+_\gamma.
\end{aligned}
\end{equation}
For the first term on the right-hand side of \eqref{eq:decom}, we have
\begin{align*}
    P^+_\gamma(P^+_{\gamma_{n-1}}-P^+_{\gamma_{n-2}})P^+_{\gamma_{n-2}}\gamma_{n-1}P^+_{\gamma_{n-1}}P^+_\gamma=I_1+I_2
\end{align*}
where thanks to \eqref{eq:dP},
\begin{align*}
    &I_1:=P^+_{\gamma_{n-2}}(P^+_{\gamma_{n-1}}-P^+_{\gamma_{n-2}}-dP^+_{\gamma_{n-2}}(\gamma_{n-1}-\gamma_{n-2}))P^+_{\gamma_{n-2}}\gamma_{n-1}P^+_{\gamma_{n-1}}P^+_\gamma,\\
    &I_2:=(P^+_\gamma-P^+_{\gamma_{n-2}})(P^+_{\gamma_{n-1}}-P^+_{\gamma_{n-2}})P^+_{\gamma_{n-2}}\gamma_{n-1}P^+_{\gamma_{n-1}}P^+_\gamma.
\end{align*}
Then from \eqref{eq:P-P-dP} and \eqref{eq:DP}, we infer
\begin{align*}
  \|I_1\|_{X_c}&\leq \frac{(1+\kappa)^{1/2}}{(1-\kappa)^{1/2}}\|\gamma_{n-1}P^+_{\gamma_{n-1}}P^+_\gamma|\D|^{1/2}\|_{\sigma_1}\||\D|^{1/2}(P^+_{\gamma_{n-1}}-P^+_{\gamma_{n-2}}-dP^+_{\gamma_{n-2}}(\gamma_{n-1}-\gamma_{n-2}))\|_{\mathcal{B}(\mathcal{H})}\\
    &\leq\frac{\pi^2(1+\kappa)^{3/2}}{8(1-\kappa)^{2}\lambda_0^{3/2}}\frac{\alpha_c^2}{c^3}\|\gamma_{n-1}-\gamma_{n-2}\|_{X_c}^2\|\gamma_{n-1}|\D|^{1/2}\|_{\sigma_1}.
\end{align*}
According to Lemma \ref{lem:retra}, $\gamma_{n-1}\in \mathcal{U}_R$ since $\gamma\in \mathcal{U}_R$ and $T$ maps $\mathcal{U}_R$ into $\mathcal{U}_R$, and for any $n\geq 2$,
\[
\|\gamma_{n-1}-\gamma_{n-2}\|_{X_c}\leq \|\gamma-T(\gamma)\|_{X_c},\quad\|\gamma_{n-1}|\D|^{1/2}\|_{\sigma_1}\leq cR.
\]
Thus as $\kappa<1$,
\begin{align*}
    \|I_1\|_{X_c}\leq C_{\kappa,L}'\frac{R\alpha_c^2}{c^2}\|T(\gamma)-\gamma\|_{X_c}\|\gamma_{n-1}-\gamma_{n-2}\|_{X}
\end{align*}
with $C_{\kappa,L}':=\frac{\pi^2}{2(1-\kappa)^2\lambda_0^{3/2}}$. According to \eqref{eq:retra}, we have $\|\gamma_{n}-\gamma\|_{X}\leq \frac{1}{1-L}\|T(\gamma)-\gamma\|_{X_c}$. Thus, by \eqref{eq:DP} and \eqref{eq:P-P},
\begin{align*}
 \|I_2\|_{X_c}&\leq\frac{(1+\kappa)}{(1-\kappa)}\||\D|^{1/2}(P^+_\gamma-P^+_{\gamma_{n-2}})\|_{\mathcal{B}(\mathcal{H})}\||\D|^{-1/2}\|_{\mathcal{B}(\mathcal{H})}\||\D|^{1/2}(P^+_{\gamma_{n-1}}-P^+_{\gamma_{n-2}})\|_{\mathcal{B}(\mathcal{H})}\|\gamma_{n-1}|\D|^{1/2}\|_{\sigma_1}\\
&\leq\frac{\pi^2(1+\kappa)}{16(1-\kappa)^{2}\lambda_0^{3/2}}\frac{R\alpha_c^2}{c^2}\|\gamma_{n-2}-\gamma\|_{X_c}\|\gamma_{n-1}-\gamma_{n-2}\|_{X_c}\\
&\leq C_{\kappa,L}''\frac{R\alpha_c^2}{c^2}\|T(\gamma)-\gamma\|_{X_c}\|\gamma_{n-1}-\gamma_{n-2}\|_{X_c}
\end{align*}
with $C_{\kappa,L}'':=\frac{\pi^2}{8(1-\kappa)^{2}\lambda_0^{3/2}(1-L)}$. Thus,
\begin{align*}
   \MoveEqLeft \| P^+_\gamma(P^+_{\gamma_{n-1}}-P^+_{\gamma_{n-2}})P^+_{\gamma_{n-2}}\gamma_{n-1}P^+_{\gamma_{n-1}}P^+_\gamma\|_{X_c}\leq (C_{\kappa,L}'+C_{\kappa,L}'')\frac{R\alpha_c^2}{c^2}\|T(\gamma)-\gamma\|_{X_c}\|\gamma_{n-1}-\gamma_{n-2}\|_{X_c}.
\end{align*}
The term $P^+_\gamma\gamma_{n-1}P^+_{\gamma_{n-2}}(P^+_{\gamma_{n-1}}-P^+_{\gamma_{n-2}})P^+_\gamma$ in \eqref{eq:decom}
can be treated analogously:
\begin{align*}
    \| P^+_\gamma\gamma_{n-1}P^+_{\gamma_{n-2}}(P^+_{\gamma_{n-1}}-P^+_{\gamma_{n-2}})P^+_\gamma\|_{X_c}\leq (C_{\kappa,L}'+C_{\kappa,L}'')\frac{R\alpha_c^2}{c^2}\|T(\gamma)-\gamma\|_{X_c}\|\gamma_{n-1}-\gamma_{n-2}\|_{X_c}.
\end{align*}
Hence \eqref{eq:inter}. Then \eqref{theta-T+} follows with $C_{\kappa,L}:=\frac{5\pi^2}{4(1-\kappa)^2\lambda_0^{3/2}(1-L)^2}\geq  2(1-L)^{-1}(C_{\kappa,L}'+C_{\kappa,L}'')$.

We consider now the term $P^-_\gamma\theta(\gamma)P^-_\gamma$. As $\theta(\gamma)=P^+_{\theta(\gamma)}\theta(\gamma)P^+_{\theta(\gamma)}$, we have
\begin{align*}
    P^-_\gamma\theta(\gamma)P^-_\gamma= P^-_{\gamma}(P^+_{\theta(\gamma)}-P^+_{\gamma})\theta(\gamma)(P^+_{\theta(\gamma)}-P^+_{\gamma})P^-_{\gamma}.
\end{align*}
Thanks to \eqref{eq:retra}, \eqref{eq:DP} and \eqref{eq:P-P},
\begin{align*}
\| P^-_\gamma\theta(\gamma)P^-_\gamma\|_{X_c}&\leq \frac{1+\kappa}{1-\kappa}\||\D|^{1/2}(P^+_{\theta(\gamma)}-P^+_{\gamma})\|_{\mathcal{B}(\mathcal{H})}^2\|\theta(\gamma)\|_{\sigma_1}\\
    &\leq\frac{\pi^2(1+\kappa)}{16(1-\kappa)^{2}\lambda_0(1-L)^2}\frac{q\alpha_c^2}{c^2}\|T(\gamma)-\gamma\|_{X_c}^2\leq C_{\kappa,L}\frac{q\alpha_c^2}{c^2}\|T(\gamma)-\gamma\|_{X_c}^2.
\end{align*}
This ends the proof.
\end{proof}

 \subsection{Estimate on \texorpdfstring{\eqref{E-E}}{}}
 We consider now the term $T(\gamma)-\gamma$ and $P^-_{\gamma}\gamma P^-_{\gamma}$.
\begin{lemma}\label{lem:T-I}
Let $g\in\Gamma_q$ and $\gamma\in \mathcal{U}_R$. If $P^+_{g}\gamma P^+_{g}=\gamma$, we have
\begin{align}\label{eq:6.7}
    \|T(\gamma)-\gamma\|_{X_c}\leq \frac{\sqrt{2}\pi}{2(1-\kappa)\lambda_0^{1/2}}R\alpha\|\gamma-g\|_{X}
\end{align}
and
\begin{align}\label{eq:P-gamma-P}
    \|P^-_{\gamma}\gamma P^-_{\gamma}\|_{X_c}\leq \frac{\pi^2}{8(1-\kappa)^2\lambda_0}q\alpha_c^2\|g-\gamma\|_{X}^2.
\end{align}
\end{lemma}
\begin{proof}
Indeed, we have
\[
T(\gamma)-\gamma=(P^+_{\gamma}-P^+_{g})\gamma P^+_{\gamma}+P^+_g \gamma (P^+_{\gamma}-P^+_g).
\]
Using \eqref{eq:DP} and \eqref{eq:P-P} again, as $\kappa<1$,
\begin{align*}
\|T(\gamma)-\gamma\|_{X_c}&\leq \frac{2(1+\kappa)^{1/2}}{(1-\kappa)^{-1/2}}\||\D|^{1/2}(P^+_\gamma-P^+_g)\|_{\mathcal{B}(\mathcal{H})}\|\gamma|\D|^{1/2}\|_{\sigma_1}\\
&\leq \frac{\sqrt{2}\pi}{2(1-\kappa)\lambda_0^{1/2}}R\alpha\|\gamma-g\|_{X}.
\end{align*}

Concerning the second one, we have
\[
P^-_{\gamma}\gamma P^-_{\gamma}=P^-_{\gamma}(P^+_{g}-P^+_{\gamma})\gamma (P^+_g-P^+_{\gamma})P^-_\gamma.
\]
Then
\begin{align*}
 \|P^-_{\gamma}(T(\gamma)-\gamma)P^-_{\gamma}\|_{X_c}&\leq \frac{1+\kappa}{1-\kappa}\||\D|^{1/2}(P^+_{g}-P^+_{\gamma})\|_{\mathcal{B}(\mathcal{H})}^2\|\gamma\|_{\sigma_1}\\
    &\leq \frac{\pi^2}{8(1-\kappa)^2\lambda_0}q\alpha_c^2\|g-\gamma\|_{X}^2.
\end{align*}
This ends the proof.
\end{proof}
Inserting this lemma into Lemma \ref{lem:error-bound}, we can get immediately
\begin{align*}
     |E(\gamma)-\mathcal{E}(\gamma)|&\leq C_{\kappa,L} \frac{(48+8\pi) R + (48+3\pi^2 C_{\kappa,L}^{-1}) q}{8(1-\kappa)^2\lambda_0}\alpha_c^2\|g-\gamma\|_X^2\\
     &\leq C_{\kappa,L} \frac{(6+\pi)(R+q)\alpha_c^2}{(1-\kappa)^2\lambda_0}\|g-\gamma\|_X^2 \leq \frac{5\pi^2(6+\pi)}{4(1-\kappa)^4\lambda_0^{5/2}(1-L)^2}( R+q)\alpha_c^2\|g-\gamma\|_X^2.
\end{align*}
Here we use the fact that $R< \frac{1}{2a(\alpha,c)}\leq \frac{2}{\pi\alpha_c}$ and $C_{\kappa,L}^{-1}\leq \frac{4}{5\pi^2}$. This gives \eqref{E-E}.

\subsection{Estimate on \texorpdfstring{\eqref{E-E'}}{}}

In the above proof of \eqref{E-E}, one of the most important ingredients is Eqn. \eqref{eq:P-P}, i.e.,
\begin{align*}
        \||\D|^{1/2}(P^+_{\gamma}-P^+_{\gamma'})\|_{\mathcal{B}(\mathcal{H})}\leq a(\alpha,c)\|\gamma-\gamma'\|_X\leq \frac{a(\alpha,c)}{c}\|\gamma-\gamma'\|_{X_c}.
    \end{align*}
In order to get a better estimate on the error bound of the DF energy and the DF functional, we study the estimate on $ \||\D|^{1/2}(P^+_{\gamma}-P^+_{\gamma'})\|_{\mathcal{B}(\mathcal{H})}$ more delicately under the condition $g,\gamma\in Y$.

Before going further, we need the following.
\begin{lemma}\label{lem:6.5}
Let $h\in Y$ and $\gamma\in \Gamma_q$, then for any $u\in \mathcal{H}$,
\begin{align}\label{eq:DW}
    \|[W_{h},\D_{\gamma}-c^2\beta]u\|_{\mathcal{H}}\leq 16c(1+\kappa)\|h\|_Y\|u\|_{\mathcal{H}}.
\end{align}
\end{lemma}
\begin{proof}
As $h\in Y$, the term $(1-\Delta)^{1/2}h (1-\Delta)^{1/2}$ can be written as
\begin{align*}
    (1-\Delta)^{1/2}h (1-\Delta)^{1/2}=\sum_{k=1}^\infty\mu_k\left|\phi_k\right>\left<\phi_k\right|
\end{align*}
where $(\phi_k)_{k\geq 1}$ is an orthonormal basis on $\mathcal{H}$, $\mu_k\in\mathbb{R}$ and $\sum_{j=1}^\infty|\mu_n|=\|h\|_Y$. Then
\begin{align*}
    h=\sum_{k=1}^\infty \mu_k\left|\widetilde{\phi}_k\right>\left<\widetilde{\phi}_k\right|,
\end{align*}
with $\widetilde{\phi}_k=(1-\Delta)^{-1/2}\phi_k$. It suffices to show that for any $k\geq 1$, $\|[W_{\left|\widetilde{\phi}_k\right>\left<\widetilde{\phi}_k\right|},\D_{\gamma}]u\|_{\mathcal{H}}\leq 10(1+2\kappa)\|u\|_{\mathcal{H}}$.

We write $W_{\gamma}=W_{1,\gamma}+W_{2,\gamma}$ where for any $u\in \mathcal{H}$,
\begin{align*}
    W_{1,\gamma}u=\rho_{\gamma}*Wu=\int_{\mathbb{R}^3}\frac{\rho_{\gamma}(x-y)u(x)}{|y|}dy,\quad W_{2,\gamma}u=\int_{\mathbb{R}^3}\frac{\gamma(x,y)u(y)}{|x-y|}dy.
\end{align*}
Then $[W_{\left|\widetilde{\phi}_k\right>\left<\widetilde{\phi}_k\right|},\D_{\gamma}-c^2\beta]=[W_{1,\left|\widetilde{\phi}_k\right>\left<\widetilde{\phi}_k\right|},\D_{\gamma}-c^2\beta]+[W_{2,\left|\widetilde{\phi}_k\right>\left<\widetilde{\phi}_k\right|},\D_{\gamma}-c^2\beta]$. We study them separately. For the term $[W_{1,\left|\widetilde{\phi}_k\right>\left<\widetilde{\phi}_k\right|},\D_{\gamma}]$, we have
\begin{align*}
    [W_{1,\left|\widetilde{\phi}_k\right>\left<\widetilde{\phi}_k\right|},\D_{\gamma}-c^2\beta]u=ic[\sum_{j=1}^3\alpha_j\partial_j, W_{1,\left|\widetilde{\phi}_k\right>\left<\widetilde{\phi}_k\right|}]u+\alpha[W_{1,\left|\widetilde{\phi}_k\right>\left<\widetilde{\phi}_k\right|},W_{2,\gamma}]u.
\end{align*}
By Hardy inequality,
\begin{align*}
   \|[\sum_{j=1}^3\alpha_j\partial_j, W_{1,\left|\widetilde{\phi}_k\right>\left<\widetilde{\phi}_k\right|}]u\|_{\mathcal{H}}&= \|\int\frac{[\sum_{j=1}^3\alpha_j\partial_j |\widetilde{\phi}_k|^2](x-y)u(x)}{|y|}dy\|_{\mathcal{H}}\\
    &\leq 2\|\nabla\widetilde{\phi}_k\|_{\mathcal{H}}\left\|\left(\int|y|^{-2}|\widetilde{\phi}_k|^2(\cdot-y)dy\right)^{1/2}\!\!u \right\|_{\mathcal{H}}\leq 4\|\nabla \widetilde{\phi}_k\|_{\mathcal{H}}^2\|u\|_{\mathcal{H}}\leq 4\|u\|_{\mathcal{H}}.
\end{align*}
On the other hand, we also have
\begin{align*}
    \left|\left<v,[W_{1,\left|\widetilde{\phi}_k\right>\left<\widetilde{\phi}_k\right|},W_{2,\gamma}]u\right>\right|&=\left|\;\iiint\limits_{\mathbb{R}^{3\times 3}}\int\limits_{t\in [0,1]}\!\!\!\!\!\! \frac{\nabla|\widetilde{\phi}_k|^2(y+t(x-y)-z)\cdot (x-y)}{|z|}\frac{v^*(x)\gamma(x,y)u(y)}{|x-y|}\;dtdxdydz\right|\\
    &\leq 4\|\nabla \widetilde{\phi}_k\|_{\mathcal{H}}^2\iint_{\mathbb{R}^{3\times 2}}(|v|(x)\rho_{\gamma}^{1/2}(y))(|u|(y)\rho_{\gamma}^{1/2}(x))dxdy\leq 4q\|u\|_{\mathcal{H}}\|v\|_{\mathcal{H}}.
\end{align*}
Thus,
\begin{align}\label{eq:W1}
     \|[W_{1,\left|\widetilde{\phi}_k\right>\left<\widetilde{\phi}_k\right|},\D_{\gamma}-c^2\beta]u\|_{\mathcal{H}}\leq 4c(1+\alpha_c q)\|u\|_{\mathcal{H}}\leq 4c(1+\kappa)\|u\|_{\mathcal{H}}.
\end{align}

Now we turn to the term $[W_{2,\left|\widetilde{\phi}_k\right>\left<\widetilde{\phi}_k\right|},\D_{\gamma}-c^2\beta]$. As $\|A\|_{\mathcal{B}(\mathcal{H})}=\|A^*\|_{\mathcal{B}(\mathcal{H})}$ and using \eqref{eq:5.1-2} and the Hardy inequality, we have
\begin{align*}
    \|[W_{2,\left|\widetilde{\phi}_k\right>\left<\widetilde{\phi}_k\right|},\D_{\gamma}-c^2\beta]\|_{\mathcal{B}(\mathcal{H})}\leq 2c(1+\kappa)\|\nabla W_{2,\left|\widetilde{\phi}_k\right>\left<\widetilde{\phi}_k\right|}\|_{\mathcal{B}(\mathcal{H})}+c^2\|[W_{2,\left|\widetilde{\phi}_k\right>\left<\widetilde{\phi}_k\right|},\beta]\|_{\mathcal{B}(\mathcal{H})}.
\end{align*}
Notice that
\begin{align*}
    \nabla (W_{2,\left|\widetilde{\phi}_k\right>\left<\widetilde{\phi}_k\right|}u) =\int\frac{(x-y) \widetilde{\phi}_k(x)\widetilde{\phi}_k^*(y)u(y)}{|x-y|^3}dy-\int\frac{\nabla \widetilde{\phi}_k(x)\widetilde{\phi}_k^*(y)u(y)}{|x-y|}dy.
\end{align*}
Then we have
\begin{align*}
\left|\int\frac{ v^*(x)\nabla\widetilde{\phi}_k(x)\widetilde{\phi}_k^*(y)u(y)}{|x-y|}dx dy\right|&\leq \left(\int |\nabla\widetilde{\phi}_k(x)|^2|u(y)|^2 dxdy\right)^{1/2}\left(\int \frac{|\widetilde{\phi}_k(y)|^2|v(x)|^2}{|x-y|^2}dxdy\right)^{1/2}\\
&\leq 2\|u\|_{\mathcal{H}}\|v\|_{\mathcal{H}}
\end{align*}
and
\begin{align*}
   \MoveEqLeft\left|\int\frac{(x-y) v^*(x)\widetilde{\phi}_k(x)\widetilde{\phi}_k^*(y)u(y)}{|x-y|^3}dx dy\right|\\
   &\leq \left(\int \frac{|\widetilde{\phi}_k(x)|^2|u(y)|^2}{|x-y|^2} dxdy\right)^{1/2}\left(\int \frac{|\widetilde{\phi}_k(y)|^2|v(x)|^2}{|x-y|^2}dxdy\right)^{1/2}\leq 4\|u\|_{\mathcal{H}}\|v\|_{\mathcal{H}}.
\end{align*}
Thus,
\begin{align*}
    \|\nabla (W_{2,\left|\widetilde{\phi}_k\right>\left<\widetilde{\phi}_k\right|}u)\|_{\mathcal{H}}=\|(\nabla W_{2,\left|\widetilde{\phi}_k\right>\left<\widetilde{\phi}_k\right|})u\|_{\mathcal{H}}\leq 6\|u\|_{\mathcal{H}}
\end{align*}
from which we infer
\begin{align}\label{eq:W2}
    \|[W_{2,\left|\widetilde{\phi}_k\right>\left<\widetilde{\phi}_k\right|},\D_{\gamma}-c^2\beta]\|_{\mathcal{B}(\mathcal{H})}\leq 12c(1+\kappa).
\end{align}
This and \eqref{eq:W1} gives \eqref{eq:DW}.
\end{proof}

\begin{lemma}
Assume that $\kappa<1$. Then for any $\gamma,\gamma'\in \Gamma_q$,
\begin{align}\label{eq:P-P'}
    \||\D|^{1/2}(P^+_{\gamma}-P^+_{\gamma'})\|_{\mathcal{B}(\mathcal{H})}\leq   \frac{\alpha_c}{2 c(1-\kappa)^{1/2}\lambda_0^{3/2}}\left(16(1+\kappa)\|g-\gamma\|_Y +c\|[W_{g-\gamma},\beta]\|_{\mathcal{B}(\mathcal{H})}\right).
\end{align}
\end{lemma}
\begin{proof}
First of all, we recall that
\begin{align*}
    dP^+_{\gamma}h=\frac{\alpha}{2\pi}\int_{-\infty}^{+\infty}(\D_{\gamma}-iz)^{-1}W_{h}(\D_{\gamma}-iz)^{-1}dz.
\end{align*}
As $\sigma(|D_{\gamma}|)>0$, from the following identity
\begin{align*}
    \int_{-\infty}^{+\infty}(\D_{\gamma}-iz)^{-2}dz=0,
\end{align*}
we infer
\begin{align}\label{eq:dP2}
    dP^+_{\gamma}h&=\frac{\alpha}{2\pi}\int_{-\infty}^{+\infty}(\D_{\gamma}-iz)^{-1}[W_{h},(\D_{\gamma}-iz)^{-1}]dz\notag\\
    &=\frac{\alpha}{2\pi}\int_{-\infty}^{+\infty}(\D_{\gamma}-iz)^{-1}(\D_{\gamma}-iz)^{-1}[W_{h},\D_{\gamma}](\D_{\gamma}-iz)^{-1}dz
\end{align}
Proceeding as for \eqref{eq:P-P-dP2} and using \eqref{eq:dP2}, we can get
\begin{align}\label{eq:6.15}
 \||\D|^{1/2}dP^+_{\gamma}h\|_{\mathcal{B}(\mathcal{H})}&\leq \frac{\alpha_c}{2 c^2 (1-\kappa)^{1/2}\lambda_0^{3/2}}\|[W_{h},\D_{\gamma}]\|_{\mathcal{B}(\mathcal{H})}\notag\\
    &\leq \frac{\alpha_c}{2 c(1-\kappa)^{1/2}\lambda_0^{3/2}}\left(16(1+\kappa)\|h\|_Y +c\|[W_{h},\beta]\|_{\mathcal{B}(\mathcal{H})}\right).
\end{align}
To end the proof, it is easy to see that 
\begin{align*}
    P^+_{\gamma}-P^+_{\gamma'}=\int_{0}^1 dP^+_{\gamma'+t(\gamma-\gamma')}(\gamma-\gamma')dt.
\end{align*}
This and \eqref{eq:6.15} give \eqref{eq:P-P'}.
\end{proof}
Replacing Eqn. \eqref{eq:P-P} by Eqn. \eqref{eq:P-P'} in the proof of Lemma \ref{lem:T-I}, we obtain
\begin{lemma}\label{lem:T-I'}
Assume that $\kappa<1$. Let $g\in\Gamma_q$, $\gamma\in \mathcal{U}_R$ and $g,\gamma\in Y$. Then if $P^+_{g}\gamma P^+_{g}=\gamma$, we have
\begin{align}\label{eq:6.7'}
    \|T(\gamma)-\gamma\|_{X_c}\leq  \frac{\sqrt{2}\alpha_c R}{(1-\kappa)\lambda_0^{3/2}}\left(32\|g-\gamma\|_Y +c\|[W_{g-\gamma},\beta]\|_{\mathcal{B}(\mathcal{H})}\right)
\end{align}
and
\begin{align}\label{eq:P-gamma-P'}
   \|P^-_{\gamma}\gamma P^-_{\gamma}\|_{X_c}\leq \frac{\alpha_c^2 q}{2c^2(1-\kappa)^2\lambda_0^3}\left(32\|g-\gamma\|_Y +c\|[W_{g-\gamma},\beta]\|_{\mathcal{B}(\mathcal{H})}\right)^2.
\end{align}
\end{lemma}
Inserting these two inequalities into Lemma \ref{lem:error-bound}, we infer
\begin{align*}
     |E(\gamma)-\mathcal{E}(\gamma)|\leq \frac{5(6+\pi)}{(1-\kappa)^4\lambda_0^{9/2}(1-L)^2}( R+q)\frac{\alpha^2}{c^4}\left(32\|g-\gamma\|_Y +c\|[W_{g-\gamma},\beta]\|_{\mathcal{B}(\mathcal{H})}\right)^2.
\end{align*}
This gives \eqref{E-E'}.

\appendix

\section{Some technical estimates}\label{sec:proof1}
In this section, we list some basic estimates used in this paper taken from \cite{sere}. The difference is only because of the change of units for $Z$, $\alpha$ and $c$.
\begin{lemma}\cite[Lemma 2.6]{sere}\label{lem:ope}
Let $\gamma\in X$. 
\begin{enumerate}
    \item 
    \begin{align}\label{eq:5.1-1}
    \|W_{\gamma}\|_{\mathcal{B}(\mathcal{H})}\leq \frac{\pi}{2}\|\gamma\|_X\leq \frac{\pi}{2 c}\|\gamma\|_{X_c}    
    \end{align}
    \item 
    \begin{align}\label{eq:5.1-2}
        \|W_{\gamma}u\|_{\mathcal{H}}\leq 2\|\gamma\|_{\sigma_1}\|\nabla u\|_{\mathcal{H}}\leq \frac{2\|\gamma\|_{\sigma_1}}{c}\||\D|^{1/2}u\|_{\mathcal{H}}.
    \end{align}
    \item  Let $\gamma\in \Gamma_q$ and $\kappa(\alpha,c)<1$. Then
    \begin{align}\label{eq:5.1-3}
        (1-\kappa(\alpha,c))^2|\D|^2\leq |\D_{\gamma}|^2\leq  (1+\kappa(\alpha,c))^{2}|\D|^2.
    \end{align}
As a result,
\begin{align}\label{eq:D-D}
        (1-\kappa(\alpha,c))|\D|\leq |\D_{\gamma}|\leq  (1+\kappa(\alpha,c))|\D|.
    \end{align}
\item Let $\gamma\in \Gamma_q$, we have
\begin{equation}\label{eq:DP}
    \||\D|^{1/2}P^{\pm}_{\gamma}u\|_{\mathcal{H}}\leq \frac{(1+\kappa(\alpha,c))^{1/2}}{(1-\kappa(\alpha,c))^{1/2}}\||\D|^{1/2} u\|_{\mathcal{H}}.
\end{equation}
\item Let $\gamma\in \Gamma_q$ and $\max(q,Z)<\frac{2}{\pi/2+2/\pi}$, then
\begin{align}\label{eq:5.1-5}
    \inf|\sigma(\D_{\gamma})|\geq c^2\lambda_0(\alpha,c)=c^2(1-\max(\alpha_c q,Z_c)).
\end{align}
\end{enumerate}
\end{lemma}

Recall that $T(\gamma)=P^+_{\gamma}\gamma P^+_{\gamma}$.
\begin{lemma}\cite[Eqns. (2.13) and (2.15)]{sere}\label{lem:A2}
Assume that $\kappa(\alpha,c)<1$. Let $a(\alpha,c)=\frac{\pi}{4}\alpha_c(1-\kappa(\alpha,c))^{-1/2}\lambda_0(\alpha,c)^{-1/2}$. Then 
\begin{enumerate}
    \item For any $\gamma,\gamma'\in \Gamma_q$,
    \begin{align}\label{eq:P-P}
        \||\D|^{1/2}(P^+_{\gamma}-P^+_{\gamma'})\|_{\mathcal{B}(\mathcal{H})}\leq a(\alpha,c)\|\gamma-\gamma'\|_X\leq \frac{a(\alpha,c)}{c}\|\gamma-\gamma'\|_{X_c}.
    \end{align}
    \item For any $\gamma\in \Gamma_q$, we have $T(\gamma)\in \Gamma_q$ and
    \begin{align}\label{eq:T-T}
        \|T^2(\gamma)-T(\gamma)\|_{X_c}\leq 2a(\alpha,c)\left(\frac{1}{c}\|T(\gamma)|\D|^{1/2}\|_{\sigma_1}+\frac{a(\alpha,c) q}{2c^2}\|T(\gamma)-\gamma\|_{X_c}\right)\|T(\gamma)-\gamma\|_{X_c}.
    \end{align}
\end{enumerate}
\end{lemma}

\section{Boundedness of the eigenfunctions.}\label{sec:B}
In this paper, we also need a priori estimates on $H^1$ norms for the eigenfunctions of the DF operators $\D_{\gamma}$ and the ep-HF operator $P^+_{g}\D_{\gamma}P^+_g$. For any wave function $\psi: \mathbb{R}^3\to \mathbb{C}^4$, we split it in blocks of upper and lower components:
\begin{align}
    \psi=\begin{pmatrix}
        \psi^{(1)}\\\psi^{(2)}
    \end{pmatrix},\quad\textrm{with}\quad \psi^{(1)},\;\psi^{(2)}:\;\mathbb{R}^3\to \mathbb{C}^2.
\end{align}
For any density matrix $\gamma\in X$, we also split its kernel $\gamma(x,y)$ in the blocks:
\begin{align}
    \gamma(\cdot,\cdot)=\begin{pmatrix}
        \gamma_{1,1} & \gamma_{1,2}\\\gamma_{2,1} &\gamma_{2,2}
    \end{pmatrix},\quad\textrm{with}\quad \gamma_{i,j}:\;\mathbb{R}^3\times \mathbb{R}^3\to \mathcal{M}_2(\mathbb{C}),\quad i,j=1,2.
\end{align}
We have the followings. 
\begin{lemma}\label{lem:unibound}
Let $\gamma\in \Gamma_q$. Let $\psi$ be a normalized eigenfunction of operator $\D_{\gamma}$ with eigenvalue $-c^2\leq \lambda\leq c^2$. Under Assumption \ref{ass:abzq}, there are constants $K$ and $K'$ independent of $\alpha$ and $c$ such that 
\begin{align}\label{eq:prop:eigen1}
    \|\psi\|_{H^1}\leq K,\quad \|\psi^{(2)}\|_{\mathcal{H}}\leq \frac{K'}{c}.
\end{align}
Furthermore, for any $\gamma'\in \Gamma_q$ satisfying $0\leq \gamma' \leq \mathbbm{1}_{(0,c^2)}(\D_{\gamma})$,
\[
\|\gamma'\|_Y\leq K^2q.
\]
\end{lemma}
\begin{proof}
The proof is essentially the same as in \cite[Lemma 7 and Theorem 3]{esteban2001nonrelativistic}. As
\begin{align}\label{eq:eigen1}
    \D_{\gamma}\psi=\lambda \psi,
\end{align}
we have
\[
\|\D \psi\|_{\mathcal{H}}=\|(\lambda+\alpha V-\alpha W_{\gamma})\psi\|_{\mathcal{H}}.
\]
Thus according to the Hardy inequality and $|\lambda|\leq c^2$,
\[
c^4\|\psi\|^2_{\mathcal{H}}+c^2\|\nabla \psi\|^2_{\mathcal{H}}\leq c^4\|\psi\|^2_{\mathcal{H}}+4\alpha(Z+q)c^2\|\nabla \psi\|_{\mathcal{H}}\|\psi\|_{\mathcal{H}}+4\alpha^2(Z+q)^2\|\nabla\psi\|^2_{\mathcal{H}}.
\]
Thus,
\[
\|\nabla \psi\|_{\mathcal{H}}\leq \left(\frac{4\alpha(Z+q)}{1-\kappa(\alpha,c)
^2}\right)^{1/2}\|\psi\|_{\mathcal{H}}.
\]
As $\|\psi\|_{\mathcal{H}}=1$ and according to Remark \ref{rem:ass}, we know there is a constant $K>0$ such that
\[
\|\psi\|_{H^1}\leq K.
\]
Let $\mathcal{L}:=-i(\sigma\cdot \nabla)$. Then \eqref{eq:eigen1} can be rewritten as
\begin{align}\label{eq:decom-Dirac}
    c\mathcal{L}\psi^{(2)}-V\psi^{(1)}+\left(\rho_{\gamma}*\frac{1}{|x|}\right)\psi^{(1)}-\int_{\mathbb{R}^3}\frac{\gamma_{1,1}(x,y)\psi^{(1)}(y)+\gamma_{1,2}(x,y)\psi^{(2)}(y)}{|x-y|}dy &= (\lambda-c^2)\psi^{(1)},\notag\\
    c\mathcal{L}\psi^{(1)}-V\psi^{(2)}+\left(\rho_{\gamma}*\frac{1}{|x|}\right)\psi^{(2)}-\int_{\mathbb{R}^3}\frac{\gamma_{2,1}(x,y)\psi^{(1)}(y)+\gamma_{2,2}(x,y)\psi^{(2)}(y)}{|x-y|}dy &= (\lambda+c^2)\psi^{(2)}.
\end{align}
Dividing by $\lambda+c^2$ the second equation of \eqref{eq:decom-Dirac}, we get
\begin{align}
   \|\psi^{(2)}\|_{\mathcal{H}}\leq \frac{c}{\lambda+c^2}\|\mathcal{L}\psi^{(1)}\|_{\mathcal{H}}+\frac{C}{\lambda+c^2}\|\psi\|_{H^1}\leq \frac{K'}{c}
\end{align}

For the second estimate, according to \eqref{eq:gamma}, we rewrite $\gamma'$ as 
\[
\gamma'=\sum_{j=1}^{+\infty} \mu_j\left|\psi_j\right>\left<\psi_j\right|
\]
where $\mu_j\geq 0$, $\sum_{j=1}^{+\infty}\mu_j=q$ and $\psi_j$ is the normalized eigenfunctions of $\D_{\gamma}$ with eigenvalues $0\leq\lambda_j\leq c^2$. Thus,
\[
\|\gamma'\|_{Y}\leq \sum_{j=1}^{+\infty}\mu_j\|\psi_j\|_{H^1}^2\leq K^2q.
\]
This ends the proof.
\end{proof}

We turn to prove the boundedness of the eigenfunctions of the DF type operator $P^+_g\D_{\gamma}P^+_g$.
\begin{lemma}\label{lem:unibound'}
Let $\gamma,g\in \Gamma_q$. Let $\psi$ be a normalized eigenfunction of operator $P^+_g\D_{\gamma}P^+_g$ with eigenvalue $-c^2\leq \lambda\leq c^2$. Under Assumption \ref{ass:abzq}, there are constants $K$ and $K'$ independent of $\alpha$ and $c$ such that
\begin{align}\label{eq:prop:eigen2}
    \|\psi\|_{H^1}\leq K,\quad \|\psi^{(2)}\|_{\mathcal{H}}\leq \frac{K'}{c}.
\end{align}
Furthermore, let $\gamma'\in \Gamma^{(g)}_q$ satisfying $0\leq \gamma'\leq \mathbbm{1}_{(0,c^2)}(P^+_g\D_{\gamma}P^+_g)$, then
\[
\|\gamma'\|_Y\leq K^2q.
\]
\end{lemma}
\begin{proof}
It is easy to see that $P^+_g\psi=\psi$. As $P^+_g\D_{\gamma}P^+_g=\D_{g}P^+_g+P^+_gW_{\gamma-g}P^+_g$, we have
\begin{align}\label{eq:eigen2}
    P^+_g\D_{\gamma}P^+_g\psi=\D_{g}\psi +P^+_gW_{\gamma-g}\psi=\lambda\psi.
\end{align}
By Lemma \ref{lem:ope},
\[
\|P^+_g W_{\gamma-g}\psi\|_{\mathcal{H}}\leq \|W_{\gamma-g}\psi\|_{\mathcal{H}}\leq 4q \|\nabla \psi\|_{\mathcal{H}}.
\]
Then,
\[
\|\D \psi\|_{\mathcal{H}}=\|(\lambda+\alpha V-\alpha W_{\gamma}-P^+_gW_{\gamma-g})\psi\|_{\mathcal{H}}.
\]
Proceeding as in the proof of Lemma \ref{lem:unibound} for \eqref{eq:eigen2}, we know that under Assumption \ref{ass:abzq}, there is a constant $K>0$ such that
\[
\|\psi\|_{H^1}\leq K,\quad \|\psi^{(2)}\|_{\mathcal{H}}\leq \frac{K'}{c},\quad \textrm{and}\quad
\|\gamma'\|_Y\leq K^2q.
\]
\end{proof}

\medskip

{\bf Acknowledgments.} L.M acknowledges support from the European Research Council (ERC) under the European Union's Horizon 2020 research and innovation program (grant agreement No. 810367). The author also wishes to thank Isabelle Catto and Eric S\'er\'e for some useful discussions. 

\medskip

\bibliographystyle{plain}
\bibliography{reference}

\end{document}